\documentclass[11pt,a4paper]{article}
\usepackage[latin1]{inputenc}
\usepackage{amsmath}
\usepackage{amsthm}
\usepackage{amsfonts}
\usepackage{amsfonts,amsthm,latexsym,amsmath,amssymb,amscd,epsfig,psfrag,enumerate}
\usepackage{graphics,graphicx, bezier, float, color, hyperref}
\usepackage{amssymb,url}
\usepackage{multienum}
\usepackage[table]{xcolor}
\usepackage{multicol,multirow}
\usepackage{graphicx}
\usepackage{fancyvrb}
\usepackage{parskip}
\usepackage[toc,page]{appendix}
\usepackage[top=2.7cm, bottom=2.7cm, left=1.5cm, right=1.5cm]{geometry}
\usepackage{xcolor}

\usepackage{blkarray}
\newtheorem{theorem}{Theorem}[section]
\newtheorem{lemma}{Lemma}[section]

\usepackage[none]{hyphenat}[section]
\newtheorem{definition}{Definition}[section]

\numberwithin{equation}{section}
\numberwithin{table}{section}
\numberwithin{figure}{section}

\title{On a problem of Pillai involving $S-$units
	and Lucas numbers}
\author{Herbert Batte$^{1,*} $, Mahadi Ddamulira$^{1}$, Juma Kasozi$^{1}$, and Florian Luca$^{2}$}
\date{}

\begin{document}
\maketitle
\abstract{ Let $ \{L_n\}_{n\geq 0} $ be the sequence of Lucas numbers. In this paper, we look at the exponential Diophantine equation $L_n-2^x3^y=c$, for $n,x,y\in \mathbb{Z}_{\ge0}$. We treat the cases $c\in -\mathbb{N}$, $c=0$ and $c\in \mathbb{N}$ independently. In
the cases that  $c\in \mathbb{N}$ and $c\in -\mathbb{N}$, we find all integers $c$ such that the Diophantine equation has at least three solutions. These cases are treated independently since we employ quite different techniques in proving the two cases. } 

{\bf Keywords and phrases}: Lucas numbers; $S-$units; linear forms in logarithms; $p-$adic numbers; Pillai's problem.
 
{\bf 2020 Mathematics Subject Classification}: 11B39, 11D61, 11D45, 11Y50.

\thanks{$ ^{*} $ Corresponding author}

\section{Introduction}\label{intro}
\subsection{Background}
\label{sec:1.1}
Let $\{L_n\}_{n\ge 0}$ be the sequence of Lucas numbers given by $L_0=2$, $L_1=1$ and the recurrence relation $L_{n+2}=L_{n+1}+L_{n}$ for all $n \geq 0$.
The first few terms of this sequence are given by 
$$
2,\;1,\;3,\;4,\;7,\;11,\;18,\;29,\;47,\;76,\;123,\;199,\ldots. 
$$
The Diophantine equation
\begin{equation} \label{1.1}
	a^{x}-b^{y}=c,
\end{equation}
where fixed integers $a>1$, $b>1$, and $c$ are considered, is commonly referred to as the Pillai equation. Its exploration has historical roots dating back to Pillai, as noted in \cite{PIL1}. The question of whether equation \eqref{1.1} admits multiple solutions $(x,y)$ has intrigued mathematicians. Notably, Pillai demonstrated that when positive coprime integers $a$ and $b$ are involved, and $|c|>c_0(a,b)$, the equation admits at most a single solution $(x,y)$.

Recent extensions of the Pillai problem have examined cases where $a$ is set to $2$, $3$ or a fixed prime $p$, while the sequence of powers of $b$ is replaced by other sequences of positive integers demonstrating exponential growth, including Fibonacci numbers, Tribonacci numbers, Pell numbers, and more complex $k$-generalized Fibonacci numbers with unknown $k$. In these studies, it has been established that the essential conclusion of the original Pillai problem holds true: for most integers, a unique representation can be found, barring certain exceptions.  See, for example,  \cite{BAT}, \cite{BFY}, \cite{CPZ}, \cite{DDA} and \cite{DFR}.

In a number field \( K \) with ring of integers \( R \) and a finite set of prime ideals \( S \), an element \( x\in K \) is an $S-$unit if its principal fractional ideal is a product of primes in \( S \). For rational integers, an $S-$unit is a rational number with numerator and denominator divisible only by primes in \( S \). The most recent study of the variation of \eqref{1.1}, is in \cite{VZ}, where the powers of $a$ were replaced by members of the Fibonacci sequence and the powers of $b$ were replaced by  $S-$units.  For us, we replace the Fibonacci sequence with the sequence of Lucas numbers and keep the $S$--units with $S=\{2,3\}$.

That is, we look at the exponential Diophantine equation 
\begin{align}\label{1.2}
	L_n-2^x3^y=c,
\end{align}
for $n,x,y\in \mathbb{Z}_{\ge 0}$.
\subsection{Main Results}\label{sec:1.2}
Our results are the following.
\begin{theorem}\label{1.2a} 
	The Diophantine equation \eqref{1.2} has in the case that $c = 0$, exactly five
	solutions, namely
	$$(n,x,y)= (0, 1, 0), (1, 0, 0), (2, 0, 1),	(3, 2, 0),
	(6, 1, 2). $$
\end{theorem}

\begin{theorem}\label{1.2b} 
	Let $c\in \mathbb{N}$ such that the Diophantine equation \eqref{1.2} has at least three solutions $(n,x,y)\in \mathbb{Z}_{\ge 0}^3$. Then, $$c \in \{1,2,3,5,9,15,20,75\}.$$
 Furthermore, these representations are given below.
\begin{align*}
	1&=L_0-2^03^0=L_2-2^1 3^0=L_3-2^03^1=L_4-2^1 3^1,\\
	2&=L_2-2^03^0=L_3-2^1 3^0=L_5-2^03^2=L_6-2^4 3^0=L_7-2^03^3,\\
	3&=L_3-2^03^0=L_4-2^2 3^0=L_5-2^3 3^0,\\
	5&=L_4-2^13^0=L_5-2^1 3^1=L_7-2^3 3^1,\\
	9&=L_5-2^13^0=L_6-2^0 3^2=L_{13}-2^9 3^0,\\
	15&=L_6-2^03^1=L_8-2^5 3^0=L_{10}-2^2 3^3,\\
	20&=L_7-2^03^2=L_8-2^0 3^3=L_{16}-2^0 3^7,\\
	75&=L_9-2^03^0=L_{10}-2^4 3^1=L_{14}-2^8 3^1.
\end{align*}
\end{theorem}

\begin{theorem}\label{1.2c} 
	Let $c\in -\mathbb{N}$ such that the Diophantine equation \eqref{1.2} has at least three solutions $(n,x,y)\in \mathbb{Z}_{\ge 0}^3$. Then, $$c \in \{-133,-97,-61,-52,-25,-21,-20,-17,-14,-9,-7,-6,-5,-3,-2,-1\}.$$
 Furthermore, these representations are given below.
\begin{align*}
	-133&=L_5-2^4 3^2=L_7-2^1 3^4=L_{10}-2^8 3^0,\\
	-97&=L_5-2^2 3^3=L_8-2^4 3^2=L_{16}-2^8 3^2,\\
	-61&=L_2-2^6 3^0=L_5-2^3 3^2=L_8-2^2 3^3,\\
	-52&=L_0-2^13^3=L_7-2^0 3^4=L_9-2^7 3^0,\\
	-25&=L_0-2^03^3=L_4-2^5 3^0=L_5-2^2 3^2=L_7-2^1 3^3=L_8-2^3 3^2,\\
	-21&=L_2-2^3 3^1=L_5-2^5 3^0=L_{10}-2^4 3^2=L_{14}-2^5 3^3,\\
	-20&=L_3-2^3 3^1=L_4-2^0 3^3=L_9-2^5 3^1,\\
	-17&=L_1-2^1 3^2=L_4-2^3 3^1=L_8-2^6 3^0=L_{11}-2^3 3^3,\\
	-14&=L_0-2^4 3^0=L_3-2^1 3^2=L_6-2^5 3^0,\\
	-9&=L_2-2^2 3^1=L_4-2^4 3^0=L_6-2^0 3^3,\\
	-7&=L_0-2^0 3^2=L_1-2^3 3^0=L_5-2^1 3^2=L_7-2^2 3^2=L_8-2^1 3^3,\\
	-6&=L_0-2^3 3^0=L_2-2^0 3^2=L_6-2^3 3^1,\\
	-5&=L_1-2^13^1=L_2-2^3 3^0=L_{3}-2^0 3^2=L_4-2^2 3^1=L_5-2^4 3^0=L_9-2^0 3^4=L_{10}-2^7 3^0,\\
	-3&=L_1-2^2 3^0=L_2-2^1 3^1=L_7-2^5 3^0,\\
	-2&=L_0-2^2 3^0=L_1-2^0 3^1=L_{3}-2^1 3^1=L_4-2^0 3^2=L_{12}-2^2 3^4,\\
	-1&=L_0-2^03^1=L_1-2^1 3^0=L_{2}-2^2 3^0=L_4-2^3 3^0=L_5-2^2 3^1=L_8-2^4 3^1.
\end{align*}
\end{theorem}

\newpage
\section{Methods}
\subsection{Preliminaries}
Here, we start with the well-known Binet formula for the sequence of Lucas numbers. It is given by
\begin{align}\label{2.1}
	L_n = \alpha^n +\beta^n,~~~\text{where}~~\alpha=\dfrac{1+\sqrt{5}}{2}, ~~\beta=\dfrac{1-\sqrt{5}}{2}.
\end{align}
Note that $\beta=-\alpha^{-1}$. 

Assuming that $n\ge 10$, then by \eqref{2.1}, we have
\begin{align*}
	L_n = \alpha^n +\beta^n=\alpha^n\left(1+(-1)^n\alpha^{-2n}\right).
\end{align*}
Since $n\ge 10$, then 
\begin{align}\label{2.2}
	0.999\alpha^n<\alpha^n\left(1-\alpha^{-20}\right)\le L_n\le \alpha^n\left(1+\alpha^{-20}\right)<1.001\alpha^n.
\end{align}
Now, in the case that $c > 0$, in \eqref{1.2}, we have
\begin{align*}
	2^x 3^y=L_n -c\le L_n-1\le \alpha^n.
\end{align*}
This implies that 
\begin{align}\label{2.3}
	x,y<n\log\alpha.
\end{align}
Moreover, in the case that $c <0$, in \eqref{1.2}, we have
\begin{align*}
	0.5\alpha^n<0.999\alpha^n < L_n =2^x 3^y+c<2^x 3^y.
\end{align*}
This implies that $\alpha^n<2^{x+1}3^y$, or simply,
\begin{align}\label{2.4}
	n\log\alpha<2\max\{(x+1)\log 2, y\log 3\}.
\end{align}
We recall one additional simple fact from calculus. This is Lemma 1 in \cite{VZ}.
\begin{lemma}[Lemma 1 in \cite{VZ}]\label{lem2.1}
If $x\in \mathbb{R}$ satisfies $|x|<\frac{1}{2}$, then $|\log(1+x)|<\frac{3}{2}|x|$.	
\end{lemma}

\subsection{Linear forms in logarithms}
We use several times Baker-types lower bounds for nonzero linear forms in two or three logarithms of algebraic numbers. There are many such bounds mentioned in the literature like that of Baker and W{\"u}stholz from \cite{BW} or Matveev from \cite{MAT}. Before we can formulate such inequalities we need the notion of height of an algebraic number recalled below.

\begin{definition}\label{def2.1}
	Let $ \gamma $ be an algebraic number of degree $ d $ with minimal primitive polynomial over the integers $$ a_{0}x^{d}+a_{1}x^{d-1}+\cdots+a_{d}=a_{0}\prod_{i=1}^{d}(x-\gamma^{(i)}), $$ where the leading coefficient $ a_{0} $ is positive. Then, the logarithmic height of $ \gamma$ is given by $$ h(\gamma):= \dfrac{1}{d}\Big(\log a_{0}+\sum_{i=1}^{d}\log \max\{|\gamma^{(i)}|,1\} \Big). $$
\end{definition}
 In particular, if $ \gamma$ is a rational number represented as $\gamma:=p/q$ with coprime integers $p$ and $ q\ge 1$, then $ h(\gamma ) = \log \max\{|p|, q\} $. 
The following properties of the logarithmic height function $ h(\cdot) $ will be used in the rest of the paper without further reference:
\begin{equation}\nonumber
	\begin{aligned}
		h(\gamma_{1}\pm\gamma_{2}) &\leq h(\gamma_{1})+h(\gamma_{2})+\log 2;\\
		h(\gamma_{1}\gamma_{2}^{\pm 1} ) &\leq h(\gamma_{1})+h(\gamma_{2});\\
		h(\gamma^{s}) &= |s|h(\gamma)  \quad {\text{\rm valid for}}\quad s\in \mathbb{Z}.
	\end{aligned}
\end{equation}
With these properties, we can easily compute
\begin{align}\label{2.5}
	h\left(\alpha^x+\xi \right)=h\left(\beta^x+\xi \right)\le \dfrac{|x|\log \alpha}{2}+\log 2, ~~~\text{for}~~\xi\in \{-1,1\}.
\end{align}

A linear form in logarithms is an expression
\begin{equation}
	\label{eq:Lambda}
	\Lambda:=b_1\log \gamma_1+\cdots+b_t\log \gamma_t,
\end{equation}
where for us $\gamma_1,\ldots,\gamma_t$ are positive real  algebraic numbers and $b_1,\ldots,b_t$ are nonzero integers. We assume, $\Lambda\ne 0$. We need lower bounds 
for $|\Lambda|$. We write ${\mathbb K}:={\mathbb Q}(\gamma_1,\ldots,\gamma_t)$ and $D$ for the degree of ${\mathbb K}$.
We start with the general form due to Matveev \cite{MAT}. 

\begin{theorem}[Matveev, \cite{MAT}]
	\label{thm:Mat} 
	Put $\Gamma:=\gamma_1^{b_1}\cdots \gamma_t^{b_t}-1=e^{\Lambda}-1$. Assume $\Gamma\ne 0$. Then 
	$$
	\log |\Gamma|>-1.4\cdot 30^{t+3}\cdot t^{4.5} \cdot D^2 (1+\log D)(1+\log B)A_1\cdots A_t,
	$$
	where $B\ge \max\{|b_1|,\ldots,|b_t|\}$ and $A_i\ge \max\{Dh(\gamma_i),|\log \gamma_i|,0.16\}$ for $i=1,\ldots,t$.
\end{theorem}

We continue with $t=2$. Let $A_1,~A_2>1$ be real numbers such that 
\begin{equation}
	\label{eq:Ai}
	\log A_i\ge \max\left\{h(\gamma_i),\frac{|\log \gamma_i|}{D},\frac{1}{D}\right\}\quad {\text{\rm for}}\quad i=1,2.
\end{equation}
Put
$$
b':=\frac{|b_1|}{D\log A_2}+\frac{|b_2|}{D\log A_1}.
$$
The following  result is Corollary 2 in \cite{LMN}. 
\begin{theorem}[Laurent et al., \cite{LMN}]
	\label{thm:LMN}
	In case $t=2$, we have  
	$$
	\log |\Lambda|\ge -24.34 D^4\left(\max\left\{\log b'+0.14,\frac{21}{D},\frac{1}{2}\right\}\right)^2\log A_1\log A_2.
	$$
\end{theorem}
We also employ a $p$-adic variation of Laurent's result as established by Bugeaud and Laurent in \cite{BL}, Corollary 1. Prior to outlining their result, we require a few additional notations. 

\begin{definition}\label{def2.2}
Let \( p \) be a prime number. The \( p \)-adic valuation of an integer \( x \), denoted by $\nu_p(x)$,  is defined by
\[
\nu_p(x) := 
\begin{cases} 
	\max\{k \in \mathbb{N} : p^k \mid x\}, & \text{if } x \neq 0; \\
	\infty, & \text{if } x = 0.
\end{cases}
\]
Additionally, if $x=a/b$ is a rational number and $a,~b$ are integers, then we put 
$$\nu_p(x)=\nu_p(a)-\nu_p(b).
$$
\end{definition}

The formula for $\nu_p(x)$ when $x$ is rational given by Definition \ref{def2.2} does not depend on the representation of $x$ as a ratio of integers $a/b$. It follows easily from Definition \ref{def2.2} that if $x$ is rational, then
$$
\nu_p(x)=\text{ord}_p(x),
$$
where $\text{ord}_p(x)$ is the exponent of $p$ in the factorization of $x$. For example, $\nu_2(9/8)=-3$. Now for the algebraic number $\gamma$ in Definition \ref{def2.1}, we define
\[
\nu_p(\gamma)=\frac{\nu_p(a_d/a_0)}{d},
\]
that is, it is the $p$-adic valuation of the rational number $a_d/a_0$ divided by the degree $d$ of $\gamma$. For example, when $x$ is a rational number, we write it as $x=a_d/a_0$ with coprime integers $a_d$ and $a_0\ge 1$, then 
its minimal polynomial is $f(X)=a_0X-a_d$ has degree $1$ so 
$$
\nu_p(x)=\nu_p(a_d/a_0),
$$
consistent with Definition \ref{def2.2}. 
The \( p \)-adic valuation gives rise to the conventional absolute value. When the rational numbers \( \mathbb{Q} \) are completed with the standard absolute value, the outcome is the set of real numbers, \( \mathbb{R} \). Conversely, employing the \( p \)-adic absolute value for the completion of \( \mathbb{Q} \) yields the \( p \)-adic numbers, represented as \( \mathbb{Q}_p \).

Similarly to the preceding context, let $\gamma_{1}$ and $\gamma_{2}$ be algebraic numbers over $\mathbb{Q}$, treated as elements of the field $K_p: = \mathbb{Q}_p(\gamma_{1},\gamma_{2})$, with $D = [\mathbb{Q}_p(\gamma_{1},\gamma_{2}):\mathbb{Q}_p]$. Similar to the situation in Theorem \ref{thm:Mat} above, we must employ an adjusted height function. Specifically, we express it as follows:
\[h'(\gamma_{i}) \ge \max \left\{h(\gamma_{i}),\dfrac{\log p}{D}\right\}, ~~\text{for }~i=1,2.\]
\begin{lemma}[Bugeaud and Laurent, \cite{BL}]\label{lem:Bug}
	Let $b_1$, $b_2$ be positive integers and suppose that $\gamma_{1}$ and $\gamma_{2}$ are multiplicatively independent algebraic numbers such that $v_p(\gamma_{1})=v_p(\gamma_{2})=0$. Put 
	$$  E':=\dfrac{b_1}{h'(\gamma_{2})}+\dfrac{b_2}{h'(\gamma_{1})},   $$
	and 
	$$ E:=\max\left\{\log E'+\log\log p+0.4, 10, 10\log p\right\}.    $$
	Then
	$$ v_p\left(\gamma_{1}^{b_1}\gamma_{2}^{b_2}-1\right)\le \dfrac{24pg}{(p-1)(\log p)^4} E^2D^4   h'(\gamma_{1})h'(\gamma_{2}  ), $$
	where $g>0$ is the smallest integer such that $v_p\left(\gamma_{i}^g-1\right)>0$.
\end{lemma}
To apply Lemma \ref{lem:Bug}, we have to ensure that $\gamma_{1}$ and $\gamma_{2}$ are multiplicatively independent. For us, to apply Lemma \ref{lem:Bug}, we have to ensure that $\alpha$ and 
\begin{equation}
\label{eq:tau}
\tau(t)=\frac{\alpha^t-1}{\beta^t-1}
\end{equation}
are multiplicatively independent. The following result which is Lemma 5 in \cite{VZ} is helpful.
\begin{lemma}[Lemma 5 in \cite{VZ}]\label{lem2.3}
	Let $t \ge 1$ be an integer. The algebraic numbers $\alpha$ and $\tau(t)$ are	multiplicatively dependent if and only if $t = 1, 3$ or $t$ is even. In these cases we have
	$$\tau(1)=-\alpha^{-2}, \qquad \tau(3)=-\alpha^2\qquad {\text{\rm and}} \qquad \tau(2t)=-\alpha^{2t}.$$
\end{lemma}

Lemma \ref{lem2.3} shows that in infinitely many cases ($t$ is even or $t=1,3$), Lemma \ref{lem:Bug} is not applicable and we have to estimate $v_p\left(\alpha^x\pm 1\right)$. This can be done by using the so-called $p-$adic logarithm $\log _p$ (see \cite{SMA}, Sect. II.2.4. for a more detailed approach).

For an algebraic number $x$ let $| x|_p:=p^{-\nu_p(x)}$. 
Let $\mathbb{C}_p$ be the complex $p$-adic field which is complete with respect to $|\cdot|_p$  and also algebraically closed. Let $\log_p x$ be the $p$-adic logarithm defined on
defined in the ball 
$$
D(0,p^{-\frac{1}{p-1}}):=\{\xi\in {\overline{\mathbb Q}}_p: |\xi-1 |_p<p^{-\frac{1}{p-1}}\}
$$
given by
\[
\log_p \xi := -\sum_{i=1}^{\infty} \dfrac{(1-\xi)^i}{i}.
\]
The series from the right--hand side converges $p$-adically. The $p$-adic logarithm has the expected property that
$$
\log _p xy=\log _p x+\log_p y
$$
for $x,y\in D(0,p^{-\frac{1}{p-1}})$. 
Moreover, it holds that
\begin{align}\label{2.8}
	|\log_p \xi|_p = |\xi - 1|_p, \quad \text{and similarly} \quad v_p(\log_p \xi) = v_p(\xi - 1),
\end{align}
as long as $\xi\in D(0,p^{-\frac{1}{p-1}})$. The following result corresponds to Lemma 6 in \cite{VZ}.
\begin{lemma}[Lemma 6 in \cite{VZ}]\label{lem2.4}
	Let $x\ge 0$ be an integer. Then
	\begin{align*}
		v_2\left(\alpha^x-1\right)&= \left\{ \begin{array}{c}
			1+v_2(x) \quad  ;~~ \text{if}\quad  3|x\\
			0 \quad ;~~ \text{if}\quad 3\not|x
		\end{array}
		\right.,\\
		v_2\left(\alpha^x+1\right)&= \left\{ \begin{array}{c}
			1 \quad\quad ;~~ \text{if}\quad  3|x\\
			0 \quad ;~~ \text{if}\quad 3\not|x
		\end{array}
		\right.,\\
		v_3\left(\alpha^x-1\right)&= \left\{ \begin{array}{c}
			1+v_3(x) \quad  ;~~ \text{if}\quad  8|x\\
			0 \quad ;~~ \text{if}\quad 8\not|x
		\end{array}
		\right.,\\
		v_3\left(\alpha^x+1\right)&= \left\{ \begin{array}{c}
			1+v_3(x) \quad  ;~~ \text{if}\quad  x\equiv 4 \mod 8\\
			0 \quad ;~~ \text{if}\quad x\not\equiv 4 \mod 8
		\end{array}
		\right..
	\end{align*}
\end{lemma}
However, during the calculations, upper bounds on the variables which are too large are obtained, thus there is need to reduce them. In this paper we use the following result related with continued fractions (see Theorem 8.2.4 in \cite{ME}).

\begin{lemma}[Legendre]\label{lem:Leg} Let $ \mu $ be an irrational number, $[a_0,a_1,a_2,\ldots]$ be the continued fraction expansion of $\mu$. Let $p_i/q_i=[a_0,a_1,a_2,\ldots,a_i]$, for all $i\ge 0$, be all the convergents of the continued fraction of $ \mu$, and $M$ be a positive integer. Let $ N $ be a non-negative integer such that
	$ q_{N} > M $.
	Then putting $ a(M) := \max \{a_{i}: i=0,1,2,\ldots,N   \}$, the inequality
	$$ \bigg| \mu-\frac{r}{s}   \bigg| > \dfrac{1}{(a(M)+2)s^2},  $$
	holds for all pairs $ (r, s) $ of positive integers with $ 0 < s < M $. 
\end{lemma}
However, since there are no methods based on continued fractions to find a lower bound for linear forms in more than two variables with bounded integer coefficients, we use at some point a method based on the LLL--algorithm. We next explain this method.

\subsection{Reduced Bases for Lattices and LLL-reduction methods}\label{sec2.3}

Let \( k \) be a positive integer. A subset \( \mathcal{L} \) of the \( k \)-dimensional real vector space \( \mathbb{R}^k \) is called a lattice if there exists a basis \( \{b_1, b_2, \ldots, b_k \}\) of \( \mathbb{R}^k \) such that
\begin{align*}
\mathcal{L} = \sum_{i=1}^{k} \mathbb{Z} b_i = \left\{ \sum_{i=1}^{k} r_i b_i \mid r_i \in \mathbb{Z} \right\}.
\end{align*}
In this situation we say that \( b_1, b_2, \ldots, b_k \) form a basis for \( \mathcal{L} \), or that they span \( \mathcal{L} \). We
call \( k \) the rank of \( \mathcal{L} \). The determinant \( \text{det}(\mathcal{L}) \), of \( \mathcal{L} \) is defined by
\begin{align*}
	 \text{det}(\mathcal{L}) = | \det(b_1, b_2, \ldots, b_k) |,
\end{align*}
with the \( b_i \) being written as column vectors. This is a positive real number that does not depend on the choice of the basis (see \cite{Cas} Sect 1.2).

Given linearly independent vectors \( b_1, b_2, \ldots, b_k \) in \( \mathbb{R}^k \), we refer back to the Gram-Schmidt orthogonalization technique. This method allows us to inductively define vectors \( b^*_i \) (with \( 1 \leq i \leq k \)) and real coefficients \( \mu_{i,j} \) (for \( 1 \leq j \leq i \leq k \)). Specifically,
\begin{align*}
	b^*_i &= b_i - \sum_{j=1}^{i-1} \mu_{i,j} b^*_j,~~~
  \mu_{i,j} = \dfrac{\langle b_i, b^*_j\rangle }{\langle b^*_j, b^*_j\rangle},
\end{align*}
where \( \langle \cdot , \cdot \rangle \)  denotes the ordinary inner product on \( \mathbb{R}^k \). Notice that \( b^*_i \) is the orthogonal projection of \( b_i \) on the orthogonal complement of the span of \( b_1, \ldots, b_{i-1} \), and that \( \mathbb{R}b_i \) is orthogonal to the span of \( b^*_1, \ldots, b^*_{i-1} \) for \( 1 \leq i \leq k \). It follows that \( b^*_1, b^*_2, \ldots, b^*_k \) is an orthogonal basis of \( \mathbb{R}^k \). 
\begin{definition}
The basis \( b_1, b_2, \ldots, b_n \) for the lattice \( \mathcal{L} \) is called reduced if
\begin{align*}
	\| \mu_{i,j} \| &\leq \frac{1}{2}, \quad \text{for} \quad 1 \leq j < i \leq n,~~
	\text{and}\\
	\|b^*_{i}+\mu_{i,i-1} b^*_{i-1}\|^2 &\geq \frac{3}{4}\|b^*_{i-1}\|^2, \quad \text{for} \quad 1 < i \leq n,
\end{align*}
where \( \| \cdot \| \) denotes the ordinary Euclidean length. The constant \( \frac{3}{4} \) above is arbitrarily chosen, and may be replaced by any fixed real number \( v \) with \( \frac{1}{4} < v < 1 \), (see \cite{LLL} Sect 1).
\end{definition}\noindent
Let $\mathcal{L}\subseteq\mathbb{R}^k$ be a $k-$dimensional lattice  with reduced basis $b_1,\ldots,b_k$ and denote by $B$ the matrix with columns $b_1,\ldots,b_k$. 
We define
\[
l\left( \mathcal{L},v\right)= \left\{ \begin{array}{c}
	\min_{u\in \mathcal{L}}||u-v|| \quad  ;~~ v\not\in \mathcal{L}\\
\min_{0\ne u\in \mathcal{L}}||u|| \quad  ;~~ v\in \mathcal{L}
\end{array}
\right.,
\]
where $||\cdot||$ denotes the Euclidean norm on $\mathbb{R}^k$. It is well known that, by applying the
LLL-algorithm, it is possible to give in polynomial time a lower bound for $l\left( \mathcal{L},v\right)\ge c_1$ (see \cite{SMA}, Sect. V.4).
\begin{lemma}\label{lem2.5}
	Let $v\in\mathbb{R}^k$ and $z=B^{-1}v$ with $z=(z_1,\ldots,z_k)^T$. Furthermore, 
	\begin{enumerate}[(i)]
		\item if $v\not \in \mathcal{L}$, let $i_0$ be the largest index such that $z_{i_0}\ne 0$ and put $\sigma:=\{z_{i_0}\}$, where $\{\cdot\}$ denotes the distance to the nearest integer.
		\item if $v\in \mathcal{L}$, put $\sigma:=1$.
	\end{enumerate}
Finally, we have 
\[
c_1:=\max\limits_{1\le j\le k}\left\{\dfrac{||b_1||}{||b_j^*||}\right\}
\qquad\text{and}\qquad
c_2:= c_1^{-1}\sigma||b_1||.
\]
\end{lemma}

In our application, we are given real numbers $\eta_0,\eta_1,\ldots,\eta_k$ which are linearly independent over $\mathbb{Q}$ and two positive constants $c_3$ and $c_4$ such that 
\begin{align}\label{2.9}
|\eta_0+a_1\eta_1+\cdots +a_k \eta_k|\le c_3 \exp(-c_4 H),
\end{align}
where the integers $a_i$ are bounded as $|a_i|\le A_i$ with $A_i$ given upper bounds for $1\le i\le k$. We write $A_0:=\max\limits_{1\le i\le k}\{A_i\}$. 

The basic idea in such a situation, from \cite{Weg}, is to approximate the linear form \eqref{2.9} by an approximation lattice. So, we consider the lattice $\mathcal{L}$ generated by the columns of the matrix
$$ \mathcal{A}=\begin{pmatrix}
	1 & 0 &\ldots& 0 & 0 \\
	0 & 1 &\ldots& 0 & 0 \\
	\vdots & \vdots &\vdots& \vdots & \vdots \\
	0 & 0 &\ldots& 1 & 0 \\
	\lfloor C\eta_1\rfloor & \lfloor C\eta_2\rfloor&\ldots & \lfloor C\eta_{k-1}\rfloor& \lfloor C\eta_{k} \rfloor
\end{pmatrix} ,$$
where $C$ is a large constant usually of the size of about $A_0^k$ . Let us assume that we have an LLL--reduced basis $b_1,\ldots, b_k$ of $\mathcal{L}$ and that we have a lower bound $l\left(\mathcal{L},v\right)\ge c_1$ with $v:=(0,0,\ldots,-\lfloor C\eta_0\rfloor)$. Note that $ c_2$ can be computed by using the results of Lemma \ref{lem2.5}. Then, with these notations the following result  is Lemma VI.1 in \cite{SMA}.
\begin{lemma}[Lemma VI.1 in \cite{SMA}]\label{lem2.6}
	Let $S:=\displaystyle\sum_{i=1}^{k-1}A_i^2$ and $T:=\dfrac{1+\sum_{i=1}^{k}A_i}{2}$. If $c_2^2\ge T^2+S$, then inequality \eqref{2.9} implies that we either have $a_1=a_2=\cdots=a_{k-1}=0$ and $a_k=-\dfrac{\lfloor C\eta_0 \rfloor}{\lfloor C\eta_k \rfloor}$, or
	\[
	H\le \dfrac{1}{c_4}\left(\log(Cc_3)-\log\left(\sqrt{c_2^2-S}-T\right)\right).
	\]
\end{lemma}
Finally, we present an analytic argument which is Lemma 7 in \cite{GL}. 
\begin{lemma}[Lemma 7 in \cite{GL}]\label{Guz} If $ m \geq 1 $, $T > (4m^2)^m$ and $T > \displaystyle \frac{z}{(\log z)^m}$, then $$z < 2^m T (\log T)^m.$$	
\end{lemma}
SageMath 9.5 is used to perform all the computations in this work.

\subsection{Bounds for solutions to $S-$unit equations}
The purpose of this subsection is to deduce a result from the following lemma. The following lemma is Proposition 1 from \cite{VZ}. 
\begin{lemma}[Proposition 1, \cite{VZ}]\label{lem2.8}
	Let $\Delta>10^{80}$ be a fixed integer and assume that
	\begin{align}\label{2.10}
		2^x3^y-2^{x_1}3^{y_1}=\Delta.
	\end{align}
Then $2^x3^y<\Delta (\log\Delta)^{60\log\log\Delta}$.
\end{lemma}

We now state and prove the following consequence of Lemma \ref{lem2.8}.
\begin{lemma}\label{lem2.9}
	Assume that $(n,n_1,x,x_1,y,y_1)$ is a solution to $L_n-2^x3^y=L_{n_1}-2^{x_1}3^{y_1}$, with $n>500$ and $n>n_1$, then, for $X=x\log2+y\log 3$, we have
	\[
0.38\alpha^n<\exp(X)<2\alpha^n\left(n\log \alpha\right)^{60\log (n\log\alpha)},
	\]
	and
	\[
	n\log\alpha<X<2+n\log\alpha+60\left(\log (n\log\alpha)\right)^2.
	\]
\end{lemma}
\begin{proof}
	If $n>500$, then
	\[
\Delta=2^x3^y-2^{x_1}3^{y_1}=	L_n-L_{n_1}\ge L_{n-2}> L_{498}>10^{80}.
	\]
	So, we apply Lemma \ref{lem2.8}, with $	\Delta=L_n-L_{n_1}<L_n<1.001\alpha^n$, by \eqref{2.2}. This yields
	\begin{align}\label{2.11}
	\exp(X)&<1.001\alpha^n\left(\log 1.001\alpha^n\right)^{60\log\log1.001\alpha^n}\nonumber\\
	&<2\alpha^n\left(n\log \alpha\right)^{60\log (n\log\alpha)}.
	\end{align}
	To justify the last calculation, note that 
	\begin{eqnarray*}
	&& 1.001(\log(1.001\alpha^n))^{60\log\log(1.001\alpha^n)}  =  1.001 (n\log \alpha+\log(1.001))^{60\log\log(1.001\alpha^n)}\\
	& < & 1.001 (n\log \alpha)^{60\log(n\log \alpha+\log(1.001)} \left(1+\frac{\log(1.001)}{n\log \alpha}\right)^{60\log(n\log \alpha+\log(1.001))}\\
	& < & 1.001 (n\log \alpha)^{60\log(n\log \alpha)\left(1+\log(1.001)/(n\log \alpha)\right)} \exp\left(\frac{60\log(n\log \alpha+\log(1.001))\log(1.001)}{n\log\alpha}\right)\\
	& < & 1.001 (n\log \alpha)^{60\log(n\log \alpha)} (n\log \alpha)^{60\log(n\log \alpha)\log(1.001)/(n\log \alpha)} \\
	& \cdot & \exp\left(\frac{60\log(n\log \alpha+\log(1.001))\log(1.001)}{n\log\alpha}\right)\\
	& < & 1.001(n\log \alpha)^{60\log(n\log \alpha)} \exp\left(\frac{60\log(n\log \alpha)\cdot \log(n\log \alpha)\cdot \log(1.001)}{n\log \alpha}\right) \\
	& \cdot & \exp\left(\frac{60\log(n\log \alpha+\log(1.001))\log(1.001)}{n\log\alpha}\right).
	\end{eqnarray*}
In the above, we only used that $\log(1+y)<y$, which is valid for all positive real numbers $y$. 
	Inside the last two exponentials, for $n>500$, the first term is less than $0.008$ and the second is less than $0.002$. Thus, together with $1.001$, these factors account for at most
	$$1.001\cdot \exp(0.008)\cdot \exp(0.002)<2,
	$$
	which proves \eqref{2.11}. On the other hand, 
	\begin{align*}
	0.38\alpha^n<0.999\alpha^n-1.001\alpha^{n-1}<L_n-L_{n-1}<2^x3^y,
	\end{align*}
which gives
\begin{align}\label{2.12}
	0.38\alpha^n<\exp(X).
	\end{align}
Combining \eqref{2.11} and \eqref{2.12}, we get
$$
0.38\alpha^n<\exp(X)<2\alpha^n\left(n\log \alpha\right)^{60\log (n\log\alpha)}
$$
and taking logarithms both sides gives
\begin{align*}
n\log\alpha<X<\log(2/0.38)+n\log\alpha+60\left(\log (n\log\alpha)\right)^2
\end{align*}
which proves Lemma \ref{lem2.9} since $\log(2/0.38)<2$. 
\end{proof}

\section{Proof of Theorem \ref{1.2a}}
The purpose of this section is to prove Theorem \ref{1.2a}. The case $c=0$ reduces to the Diophantine equation
\begin{equation}
\label{eq:L}
L_n =2^x3^y.
\end{equation} 
To deal with it, recall that Lucas numbers $L_n$ have primitive prime factors for $n>6$. These are primes $p$ that divide $L_n$ but do not divide $L_m$ for any  positive integer $m<n$. This is a particular case of Carmichael's primitive divisor theorem proved in 1913. Since $2\mid L_3$ and $3\mid L_2$, we get that equation \eqref{eq:L} does not have solutions with $n>6$, and listing the solutions 
with $n\le 6$ one gets the list from Theorem \ref{1.2a}. 
 \qed

\section{Proof of Theorem \ref{1.2b}}
In this section, we prove Theorem \ref{1.2b}. Note that when $n\le 1500$ in \eqref{1.2}, then $x,y\le 723$, by inequality \eqref{2.3}. A computer search in SageMath 9.5 went through all $(n,x,y)$ in $0\le n\le 1500,~0\le x\le 723,~0\le y\le 723$
and selected the resulting $c$ with at least three representations of the form $L_n-2^x3^y$, and returned only the solutions given in Theorem \ref{1.2b}. This computation lasted for about 2.5 hours on an 8GB RAM Intel\textsuperscript{®}
Core\textsuperscript{TM} laptop.

\subsection{An absolute upper bound for $n$ in the case that $c \in\mathbb{N}$}\label{subsec4.1}
From now on, we assume $n>1500$. Let $n,n_1,n_2, x,x_1,x_2, y,y_1,y_2$ be non negative integers such that
\begin{align}\label{4.1}
	L_n-2^x3^y=L_{n_1}-2^{x_1}3^{y_1}=L_{n_2}-2^{x_2}3^{y_2}.
\end{align}
Note that we cannot have $n=n_1$ as this leads to $x=x_1$, $y=y_1$ and obtain the same representation of $c$. So, without loss of generality, we assume $n>n_1>n_2$.
\begin{lemma}\label{lem4.1}
	Suppose that $c>0$ has at least two representations as 
	$$c=L_n-2^x3^y=L_{n_1}-2^{x_1}3^{y_1},$$
	then $	n-n_1<4\cdot 10^{12} \log n.$
\end{lemma}
\begin{proof} In the case that $c>0$, we have $L_n>2^x3^y$ and $L_{n_1}>2^{x_1}3^{y_1}$. We go back to equation \eqref{1.2} and rewrite it as
\begin{align}
	L_n-2^x3^y&=L_{n_1}-2^{x_1}3^{y_1}\nonumber\\
	0\le c-1<c-\beta^n=\alpha^n-2^x3^y&=\alpha^{n_1}-2^{x_1}3^{y_1}+\left(\beta^{n_1}-\beta^n\right)\nonumber\\
	&\le\alpha^{n_1}-2^{x_1}3^{y_1}+\left|\beta^{n_1}-\beta^n\right|\nonumber
	\le \alpha^{n_1},\nonumber
\end{align}
where we have used the fact that $|-\beta^n+\beta^{n_1}|\le -\beta +\beta^2=1$, for all $n>1500$. So we conclude that
\begin{align}\label{4.2}
	\left|2^x 3^y\alpha^{-n}-1\right| \le \alpha^{-(n-n_1)}.
\end{align}
We now apply Theorem \ref{thm:Mat} on the left--hand side of \eqref{4.2}. Let $\Gamma:=2^x 3^y\alpha^{-n}-1=e^{\Lambda}-1$. Notice that $\Lambda\ne 0$ (in fact, it is negative). We use the field ${\mathbb K}:=\mathbb{Q}(\sqrt{5})$ of
degree $D = 2$. Here, $t := 3$,
\begin{equation}\nonumber
	\begin{aligned}
		\gamma_{1}&:=2, ~~\gamma_{2}:=3, ~~\gamma_{3}:=\alpha,\\
		b_1&:=x, ~~b_2:=y, ~~b_3:=-n.
	\end{aligned}
\end{equation} 
$B \geq \max\{|b_1|, |b_2|, |b_3|\} = \max\{x,y,n\}=n$. So, we can take $B:=n$. Also, $A_i \geq \max\{Dh(\gamma_{i}), |\log\gamma_{i}|, 0.16\} ~~~ \text{for all}~~ i=1,2,3 $. So, 
$A_1 := Dh(\gamma_{1}) = 2\log2$, $A_2 := Dh(\gamma_{2}) = 2\log3$, and $A_3 := Dh(\gamma_{3}) = 2\cdot\frac{1}{2} \log\alpha = \log \alpha$. 
Then, by Theorem \ref{thm:Mat},
	\begin{align}\label{4.3}
		\log |\Gamma| &> -1.4\cdot 30^6 \cdot 3^{4.5}\cdot 2^2 (1+\log 2)(1+\log n)(2\log 2)(2\log3)(\log \alpha)> -1.62\cdot 10^{12} \log n,
	\end{align}
	where the last inequality holds since $n>1500$. 
Comparing \eqref{4.2} and \eqref{4.3}, we get
\begin{align}\label{4.4}
	n-n_1<4\cdot 10^{12} \log n.
\end{align}
This proves Lemma \ref{lem4.1}.
\end{proof}

Next, we state and prove the following result.
\begin{lemma}\label{lem4.2}
Let	$c>0$, $X := x \log 2 +y \log 3$ and $X_1:= x_1 \log 2 + y_1 \log 3$. Then
\begin{align*}
	X-X_1 <3.6\cdot 10^{26} \left(\log n\right)^2.
\end{align*}
\end{lemma}
\begin{proof}
We now go back to  equation \eqref{1.2} and rewrite it as 
\begin{align*}
	\alpha^n -\alpha^{n_1}-2^x3^y&=-2^{x_1}3^{y_1}+\beta^{n_1}-\beta^n\\
	\dfrac{\alpha^{n_1}\left(\alpha^{n-n_1}-1\right)}{2^x3^y}-1&= \dfrac{-2^{x_1}3^{y_1}}{2^x3^y}+\dfrac{\beta^{n_1}-\beta^n}{2^x3^y}= \dfrac{-1}{\exp \left(X-X_1\right)}+\dfrac{\beta^{n_1}-\beta^n}{\exp(X)} ,
\end{align*}
and taking absolute values, we have 
\begin{align*}
	\left|\dfrac{\alpha^{n_1}\left(\alpha^{n-n_1}-1\right)}{2^x3^y}-1\right|&\le\dfrac{1}{\exp \left(X-X_1\right)}+\dfrac{1}{\exp(X)}
	\le 2\exp \left(-(X-X_1)\right),
\end{align*}
where we have used the fact that $|-\beta^n+\beta^{n_1}|\le -\beta +\beta^2=1$, for all $n>1500$. Moreover, if $X - X_1>1.4$, then $\frac{1}{2}>2\exp \left(-(X-X_1)\right)$. 

Let $\Gamma_1:=\alpha^{n_1}\left(\alpha^{n-n_1}-1\right)\cdot 2^{-x}3^{-y}-1$. Then
\begin{align}\label{4.5}
|\Gamma_1|=\left|\dfrac{\alpha^{n_1}\left(\alpha^{n-n_1}-1\right)}{2^x3^y}-1\right|
&\le 2\exp \left(-(X-X_1)\right).	
\end{align}
Notice that $\Gamma_1=e^{\Lambda_1}-1\ne 0$, otherwise we would have
\begin{align*}
	\dfrac{\alpha^{n}-\alpha^{n_1}}{2^x3^y}=1.
\end{align*}
Taking algebraic conjugates, we would get
\begin{align*}
	1=\dfrac{\beta^{n}-\beta^{n_1}}{2^x3^y}<1,
\end{align*} 
a contradiction. Therefore, $\Gamma_1\ne 0$ (so, also $\Lambda_1\ne 0$).  We use again the field $\mathbb{Q}(\sqrt{5})$ of
degree $D = 2$. Here, $t := 4$,
\begin{equation}\nonumber
	\begin{aligned}
		\gamma_{1}&:=2, ~~\gamma_{2}:=3, ~~\gamma_{3}:=\alpha, ~~\gamma_{4}:=\alpha^{n-n_1}-1,\\
		b_1&:=-x, ~~b_2:=-y, ~~b_3:=n_1, ~~b_4:=1.
	\end{aligned}
\end{equation}
Next, $B \geq \max\{|b_1|, |b_2|, |b_3|\} = \max\{x,y,n,1\}=n$, so we can take $B:=n$. Moreover, $A_i \geq \max\{Dh(\gamma_{i}), |\log\gamma_{i}|, 0.16\}$ for all $i=1,2,3 $. So, we take
$A_1 := 2\log2$, $A_2 := 2\log3$, $A_3 := \log \alpha$, as before and 
$$2h(\gamma_{4})=2h\left(\alpha^{n-n_1}-1\right)\le 2\cdot\dfrac{1}{2}(n-n_1)\log\alpha + 2\log 2<2\cdot 10^{12}\log n,$$
by \eqref{4.4}.
So, we take $A_4:=2\cdot 10^{12} \log n$. 
Then, by Theorem \ref{thm:Mat},
\begin{align}\label{4.6}
	\log |\Gamma_1| &> -1.4\cdot 30^7 \cdot 4^{4.5}\cdot 2^2 (1+\log 2)(1+\log n)(2\log 2)(2\log3)(\log \alpha)(2\cdot 10^{12}\log n)\nonumber\\
	&> -3.54\cdot 10^{26} \left(\log n\right)^2.
\end{align}
Comparing \eqref{4.5} and \eqref{4.6}, we get
\begin{align}
	X-X_1 <3.6\cdot 10^{26} \left(\log n\right)^2.
\end{align}
This proves Lemma \ref{lem4.2}.
\end{proof}

To proceed further, let us write $x_{min} := \min\{x,x_1\}$ and $y_{min} := \min\{y,y_1\}$. We state and prove the following result.
\begin{lemma}\label{lem4.3}
	Assume that $c>0$. Then either
	\[
	x_{min}, ~y_{min} < 3\cdot 10^{15}(\log n)^3,
	\]
	or $n<60000$.
\end{lemma}
\begin{proof}
Again, we go back to equation \eqref{1.2} and rewrite it as	
\begin{align*}
	\alpha^n -\alpha^{n_1}+\beta^n-\beta^{n_1}&=2^x3^y-2^{x_1}3^{y_1}\\
\dfrac{\alpha^{n_1}\left(\alpha^{n-n_1}-1\right)}{\beta^{n_1}\left(\beta^{n-n_1}-1\right)}+1&=\dfrac{2^{x_{min}}3^{y_{min}}\left(2^{x-x_{min}}3^{y-y_{min}}-2^{x_1-x_{min}}3^{y_1-y_{min}}\right)}{\beta^{n_1}(\beta^{n-n_1}-1)}.
\end{align*}	
Let us denote $2^{x-x_{min}}3^{y-y_{min}}-2^{x_1-x_{min}}3^{y_1-y_{min}}=:A$. Since $v_p(\beta)=0$, we have
	\begin{align}\label{4.8}
		v_2\left(\left(\dfrac{\alpha}{\beta}\right)^{n_1} \dfrac{\alpha^{n-n_1}-1}{\beta^{n-n_1}-1}+1 \right)&=x_{min}-v_2\left(\beta^{n-n_1}-1\right)+ v_2(A)\quad {\text{\rm or, equivalently}}\nonumber\\
		x_{min}&=v_2\left(\left(\dfrac{\alpha}{\beta}\right)^{n_1} \dfrac{\alpha^{n-n_1}-1}{\beta^{n-n_1}-1}+1 \right)+v_2\left(\beta^{n-n_1}-1\right)- v_2(A)\nonumber\\
		&\le v_2\left(\left(\dfrac{\alpha}{\beta}\right)^{n_1} \dfrac{\alpha^{n-n_1}-1}{\beta^{n-n_1}-1}+1 \right)+v_2\left(\beta^{n-n_1}-1\right),
	\end{align}
and	similarly, 
	\begin{align}\label{4.9}
	y_{min}\le v_3\left(\left(\dfrac{\alpha}{\beta}\right)^{n_1} \dfrac{\alpha^{n-n_1}-1}{\beta^{n-n_1}-1}+1 \right)+v_3\left(\beta^{n-n_1}-1\right).
\end{align}	
Next, we estimate $v_p\left(\beta^{n-n_1}-1\right)$ for $p=2,3$. By Lemma \ref{lem2.4},	
\begin{align*}
	v_p\left(\beta^{n-n_1}-1\right)&\le 1+\dfrac{\log\left(n-n_1\right)}{\log p}\\
	&< 1+\dfrac{\log\left(4\cdot 10^{12} \log n\right)}{\log p}\\
	&<1+\dfrac{5\log n}{\log p},
\end{align*}
under the assumption that $n > 1500$. Next, we estimate the first terms on the right hand side of \eqref{4.8} and \eqref{4.9}, respectively. Assuming that $n - n_1$ is even, we have, by Lemma \ref{lem2.4}	
$$
v_p\left(\left(\dfrac{\alpha}{\beta}\right)^{n_1} \dfrac{\alpha^{n-n_1}-1}{\beta^{n-n_1}-1}+1 \right)=v_p\left( \alpha^{2n_1+n-n_1}\pm 1\right)<1+\dfrac{\log 2n}{\log p}.
$$	
Therefore, we get inequalities
\[
x_{min}<1+\dfrac{\log 2n}{\log 2}+1+\dfrac{5\log n}{\log 2}<3+\frac{6\log n}{\log 2}<10\log n,
\]
and 	
\[
y_{min}<1+\dfrac{\log 2n}{\log 3}+1+\dfrac{5 \log n}{\log 3}<3+\frac{6\log n}{\log 3}<10\log n.
\]
If $n - n_1=1$, then by Lemma \ref{lem2.4}, we get
$$
v_p\left(\left(\dfrac{\alpha}{\beta}\right)^{n_1} \dfrac{\alpha^{n-n_1}-1}{\beta^{n-n_1}-1}+1 \right)=v_p\left(\alpha^{2(n-1)}\alpha^{-2}\pm 1\right)=v_p\left(\alpha^{2n-4}\pm 1\right)<1+\frac{\log 2n}{\log p},
$$
which is the same inequality as before. The same inequality is obtained when $n-n_1=3$. 
	
Assume now that $n-n_1\ge 5$ is odd. Then we can apply Lemma \ref{lem:Bug} to the first terms on the right hand side \eqref{4.8} and \eqref{4.9}, respectively. Here, we note that 
$$
\gamma_{1}:=\frac{\alpha}{\beta}\qquad {\text{\rm and}}\qquad \gamma_{2}:=\frac{\alpha^{n-n_1}-1}{\beta^{n-n_1}-1}
$$ 
are multiplicatively independent by Lemma \ref{lem2.3}. Furthermore, $h(\gamma_{1}) = 2( \log\alpha)/2 = \log \alpha$. We choose  $h'(\gamma_{1}) := \log \alpha$ in \eqref{4.8} and $h'(\gamma_{1}) := (\log 3)/2$ in \eqref{4.9}. Moreover, 
\[
h'(\gamma_{2}) \le 2h\left( \alpha^{n-n_1}-1\right) <(n-n_1)\log\alpha +\log 4<4\cdot 10^{12}\log n\log \alpha+\log 4<2\cdot10^{12}\log n,
\]
where we have used \eqref{4.4}. Therefore,
\[
E'=\dfrac{b_1}{h'(\gamma_{2})}+\dfrac{b_2}{h'(\gamma_{1})}<\dfrac{n_1}{h'(\gamma_{1})}+\dfrac{1}{h'(\gamma_{2})}<\dfrac{n}{h'(\gamma_{1})}.
\]
Assuming that $n>60000$, then 
$$
E=\max\left\{\log E'+\log\log p+0.4, 10, 10\log p\right\}<\log n - \log h'(\gamma_{1})+\log\log p+0.4< \log n+1.1
$$
in both cases.  Further, we can take $g=3$ when $p=2$ and $g=4$ when $p=3$. Therefore, \eqref{4.8} implies that
\begin{align*}
	x_{min}&\le \dfrac{24pg}{(p-1)(\log p)^4} E^2D^4   h'(\gamma_{1})h'(\gamma_{2}  )+v_2\left(\beta^{n-n_1}-1\right)\\
	&\le  \dfrac{24\cdot 2\cdot 3}{(2-1)(\log 2)^4} (\log n)^2\left(1+\frac{1.1}{\log 60000}\right)^2\cdot 2^4  ( \log\alpha )\cdot 2\cdot10^{12}\log n+1+\dfrac{5\log n}{\log 2}\\
	&< 3\cdot 10^{15}\left(\log n\right)^3.
\end{align*}
Similarly, \eqref{4.9} gives
\begin{align*}
	y_{min}&\le \dfrac{24pg}{(p-1)(\log p)^4} E^2D^4   h'(\gamma_{1})h'(\gamma_{2}  )+v_3\left(\beta^{n-n_1}-1\right)\\
	&< \dfrac{24\cdot 3\cdot 4}{(3-1)(\log 3)^4} (\log n)^2\left(1+\frac{1.1}{\log 60000}\right)^2\cdot 2^4  \left( \frac{1}{2}\log3 \right)\cdot 2\cdot10^{12}\log n+1+\dfrac{5\log n}{\log 3}\\
	&< 3\cdot 10^{15}\left(\log n\right)^3.
\end{align*}
This completes the proof of Lemma \ref{lem4.3}.
\end{proof}

Now, we consider a third solution $(n_2,x_2,y_2)$, with $n>n_1>n_2$ and we find an absolute bound for $n$. We prove the following result.
\begin{lemma}\label{lem4.4}
	If $c>0$ and $n>60000$,  then
	\[
	n<3.2\cdot 10^{33},\qquad x<2.4\cdot 10^{33}, \qquad y<1.5\cdot 10^{33} .
	\]
\end{lemma}
\begin{proof}
Lemma \ref{lem4.3} tells that out of any two solutions, the minimal $x$ and $y$ are bounded by $3\cdot  10^{15}(\log n)^3$. So, out of three solutions, at most one of them has $x$ which is not bounded by $3\cdot 10^{15}(\log n)^3$ and at most one of them has $y$ which is not bounded by the above quantity. Hence, at least one of the solutions has both $x$ and $y$ bounded by the above quantity which, in particular, shows that the minimal one satisfies
$$
 X_2=x_2\log 2+y_2\log 3<(3\log 2+3\log 3)\cdot 10^{15}(\log n)^3<6\cdot 10^{15} (\log n)^3.
 $$
With Lemma \ref{lem4.2} we get
\begin{align*}
	n\log\alpha&<X=X_2+(X_1-X_2)+(X-X_1)\\
	&< 6\cdot10^{15}(\log n)^3 + 3.6\cdot 10^{26} \left(\log n\right)^2+ 3.6\cdot 10^{26} \left(\log n\right)^2\\
	&<7.3\cdot 10^{26} \left(\log n\right)^3,
\end{align*} 	
 	which implies
 	\begin{align}\label{4.11}
 		\dfrac{n}{(\log n)^3}<1.6\cdot 10^{27}.
 	\end{align}	
 	We apply Lemma \ref{Guz} to inequality \eqref{4.11} above with $ z:=n $, $ m:=3 $, $T:=1.6\cdot10^{27}$.
 	Since $T>(4\cdot 3^2)^3=46656$, we get
 	$$n<2^m T(\log T)^m = 2^3 \cdot 1.6\cdot 10^{27}(\log 1.6\cdot 10^{27})^3 < 3.2\cdot 10^{33}.$$	
 	Further, we have by Lemma \ref{lem2.9} that
 	\begin{align*}
 		X&<2+n\log\alpha+60\left(\log (n\log\alpha)\right)^2\qquad {\text{\rm or, equivalently}}\\
 		x\log2+y\log3&<2+3.2\cdot 10^{33}\log\alpha+60\left(\log (3.2\cdot 10^{33}\log\alpha)\right)^2<1.6\cdot 10^{33}.
 	\end{align*}	
 	This gives $x<2.4\cdot 10^{33}$ and $y<1.5\cdot 10^{33}$ and completes the proof of Lemma \ref{lem4.4}.
 \end{proof}

\subsection{Reduction of the upper bound on $n$ when $c \in \mathbb{N}$}\label{subsec4.2}
In this subsection, we use the LLL--reduction method, the theory of continued fractions and as well as $p-$adic reduction methods due to \cite{PET} to find a rather small bound for $n$, which will conclude the proof of Theorem \ref{1.2b}.

To begin, we go back to equation \eqref{4.2}. Assuming $n-n_1\ge 2$, we can write
\begin{align*}
	|\Lambda|=\left|x\log 2+y\log 3-n\log\alpha\right|\le \frac{3}{2}\alpha^{-(n-n_1)},
\end{align*}
where we used Lemma \ref{lem2.1} with $n-n_1\ge 2$ since $\alpha^{n-n_1}\ge \alpha^2>2$. So, we consider the approximation lattice
$$ \mathcal{A}=\begin{pmatrix}
	1 & 0  & 0 \\
	0 & 1 & 0 \\
	\lfloor C\log 2\rfloor & \lfloor C\log 3\rfloor& \lfloor C\log\alpha \rfloor
\end{pmatrix} ,$$
with $C:= 10^{101}$ and choose $v:=\left(0,0,0\right)$. Now, by Lemma \ref{lem2.5}, we get $$|\Lambda|\ge c_1:=10^{-36} \qquad\text{and}\qquad c_2:=1.26\cdot 10^{34}.$$
Moreover, by Lemma \ref{lem4.4}, we have
\[
x<A_1:=2.4\cdot 10^{33},~~y<A_2:=1.5\cdot 10^{33},~~n<A_3:=3.2\cdot 10^{33}.
\]
So, Lemma \ref{lem2.6} gives $S=2.22\cdot 10^{66}$ and $T=1.94\cdot 10^{33}$. Since $c_2^2\ge T^2+S$, then choosing $c_3:=1.5$ and $c_4:=\log\alpha$, we get $n-n_1\le 321$.

Next, we now go back to equation \eqref{4.5}. Assume that $X-X_1\ge 100$. We can then write
\begin{align*}
	|\Lambda_1|=\left|n_1\log\alpha+\log\left(\alpha^{n-n_1}-1\right)-x\log 2-y\log 3 \right|
	&< 3\exp \left(-(X-X_1)\right),
\end{align*}
where we used Lemma \ref{lem2.1} together with the fact that $\exp(X-X_1)\ge \exp(100)>4$. So, we use the same approximation lattice
$$ \mathcal{A}=\begin{pmatrix}
	1 & 0  & 0 \\
	0 & 1 & 0 \\
	\lfloor C\log 2\rfloor & \lfloor C\log 3\rfloor& \lfloor C\log\alpha \rfloor
\end{pmatrix} ,$$
but with $C:= 10^{102}$ and choose $v:=\left(0,0,-\lfloor C\log\left(\alpha^{n-n_1}-1\right) \rfloor\right)$. It turns out that for all values $1\le n-n_1\le 321$, the chosen constant $C$ is sufficiently large, so we can still apply Lemma \ref{lem2.6}. By Lemma \ref{lem2.5}, we maintain the lower bound 
$$|\Lambda_1|\ge c_1:=10^{-36}, \quad c_2:=1.26\cdot10^{34},$$
and by Lemma \ref{lem4.4}, we also have 
\[
x<A_1:=2.4\cdot 10^{33},~~y<A_2:=1.5\cdot 10^{32},~~n_1<n<A_3:=3.2\cdot 10^{33}.
\]
So, Lemma \ref{lem2.6} still gives the same values of $S$ and $T$ as before. Since $c_2^2\ge T^2+S$, we now choose $c_3:=3$ and $c_4:=1$ and we get $X-X_1\le 157$.

Next, we go back to relations \eqref{4.8} and \eqref{4.9}; i.e.,
$$
x_{min}\le	v_2\left(\left(\dfrac{\alpha}{\beta}\right)^{n_1} \dfrac{\alpha^{n-n_1}-1}{\beta^{n-n_1}-1}+1 \right)+v_2\left(\beta^{n-n_1}-1\right),
$$
and	
$$
y_{min}\le v_3\left(\left(\dfrac{\alpha}{\beta}\right)^{n_1} \dfrac{\alpha^{n-n_1}-1}{\beta^{n-n_1}-1}+1 \right)+v_3\left(\beta^{n-n_1}-1\right).
$$	
Note that by Lemma \ref{lem2.4},	
$$
	v_p\left(\beta^{n-n_1}-1\right)\le 1+\dfrac{\log\left(n-n_1\right)}{\log p}
	< \left\{ \begin{array}{c}
		10 \quad ;~ \text{if}\quad  p=2,\\
		7 \quad ;~~ \text{if}\quad p=3.
	\end{array}
	\right.
$$
Assume that $n - n_1$ is even. Then, by Lemma \ref{lem2.3}, we have		
$$
	v_p\left(\left(\dfrac{\alpha}{\beta}\right)^{n_1} \dfrac{\alpha^{n-n_1}-1}{\beta^{n-n_1}-1}+1 \right)=v_p\left(\pm \alpha^{2n_1+n-n_1}+1\right)<1+\dfrac{\log 2n}{\log p}<\left\{ \begin{array}{c}
		114 \quad ;~~ \text{if}\quad  p=2,\\
		~72 \quad ;~~ \text{if}\quad p=3.
	\end{array}
	\right.
$$	
which gives
\[x_{min}< 124~~\text{and}~~y_{min}< 79 .\]
The cases $n - n_1\in \{1,3\}$ can be treated  similarly leading to the same upper bounds on $x_{min}$ and $y_{min}$.	

Assume now that $n-n_1\ge 5$ is odd. We explain in detail how we go about this case, that is, we indicate how to bound $\nu_p(L_n-L_{n_1})$ when $n-n_1$ is odd, $n\le 3.2\cdot 10^{33}$ and $p\in \{2,3\}$. We only do it for $p=2$, and then we automate the process in SageMath. Note that 
$$
n\le 3.2\cdot 10^{33}< 2^{112}.
$$
So, $n$ has at most $112$ binary digits. Let $d:=n-n_1\le 321$, by the results in the reduction above. So, we need an upper bound for 
$$
\nu_2(L_{n+d}-L_n),\qquad d\in [5,321],\quad {\text{\rm where}}\quad d\equiv 1\pmod 2,\quad n<3.2\cdot 10^{33}.
$$
The Lucas sequence is periodic modulo $2^{k+1}$ with period $3\cdot 2^k$. In particular, $L_{n+d}-L_n$ is periodic modulo $2^8$ with period $3\cdot 2^7<1000$. We looped over 
odd $d\in [5,321]$ for which there is $n\le 1000$ such that $2^{8}\mid L_{n+d}-L_n$. There are $80$ such $d$'s, namely
$\{5,7,9,15,17,19,29,31,33,39,\ldots,321\}$.
So, some odd numbers are missing from the above list. This means that for the missing numbers $d$ we have $\nu_2(L_{n+d}-L_n)\le 7$ always. 

Here, we will work out one $d$ only, for explanation. Namely, we take $d=17$. We calculate $n_0(d)\in [1,3\cdot 2^7]$ such that for $n=n_0(d)$ we have that $\nu_2(L_{n+d}-L_n)\ge 8$. This is unique in this case and it is $n_0(d)=71$. So, from now on, every $n\le 3.2\cdot 10^{33}$ such that $\nu_2(L_{n+d}-L_n)\ge 8$ is of the form $n=71+3\cdot 2^7 z$, where $z$ is some integer. 
And we need to find out $z$ such that $\nu_2(L_{n+17}-L_n)$ is as large as possible. For this, we go to the Binet formula and get
\begin{eqnarray*}
	L_{n+17}-L_n & = & (\alpha^{17}-1)\alpha^{71}\alpha^{3\cdot 2^7 z}+(\beta^{17}-1)\beta^{71}\beta^{3\cdot 2^7 z}\\
	& = & (\alpha^{17}-1)\alpha^{71} \exp_2((2^7 z)\log_2(\alpha^3))  +  (\beta^{17}-1)\beta^{71} \exp_2((2^7 z)\log_2 \beta^3).
\end{eqnarray*}
In the above calculation, $
\alpha^3-1=2\alpha$, so that $ |\alpha^3-1|_2=2^{-1}$.
Therefore, 
\begin{equation}
	\label{eq:1}
	\log_2 (\alpha^3)=\log_2 (1-(1-\alpha^3))=-\sum_{n\ge 1} \frac{(1-\alpha^3)^n}{n}
\end{equation}
and in the right--hand side, 
$
|(1-\alpha^3)^{n}/n|_2=2^{-(n-\nu_2(n))}\le 2^{-(n-\log n/\log 2)},
$
showing that series given in the right--hand side of \eqref{eq:1} converges. Further, looking at the first few terms, we have 
$$
1-\alpha^3=-2\alpha,\quad \frac{(1-\alpha^3)^2}{2}=2\alpha^2,\quad \nu_2\left(\frac{(1-\alpha^3)^n}{n}\right)=n-\nu_2(n)\ge 2\quad (n\ge 3),
$$
which gives that 
$$
\nu_2(\log_2(\alpha^3))=\nu_2\left(-2\alpha+2\alpha^2+\sum_{n\ge 2} \frac{(1-\alpha^3)^n}{n}\right)=\nu_2(2)=1,
$$
where we used $\alpha^2-\alpha=1$. 
For the argument of the exponential, we have 
$
\nu_2(2^7 z\log_2 (\alpha^3))\ge \nu_2(2^8z)\ge 8,
$
so $|2^7z\log 2 \alpha^3|_2\le 2^{-8}<2^{-1}$, therefore the exponential in this input is convergent $2$-adically. The same arguments work with $\alpha$ replaced by $\beta$. We now stop the argument of the logarithm at $n=120$, so put
\begin{equation}
	\label{eq:P}
	P:=-\sum_{n=1}^{120} \frac{(1-\alpha^3)^n}{n},
\end{equation}
such that
$$
\log_2 \alpha^3=P-\sum_{n\ge 121} \frac{(1-\alpha^3)^n}{n}.
$$
One checks that $n-\nu_2(n)\ge 121$ for all $n\ge 121$. Indeed, first 
$n-\log_2 n\ge n-\log n/\log 2$.
The function $n-\log n/\log 2$ is at least $121$ for all $n\ge 128$. For $n\in [121,127]$, one checks by hand that $n-\nu_2(n)\ge 121$. Thus, $
\log_2 (\alpha^3)=P+u,$
where $\nu_2(u)\ge 121$. We therefore have,
$$
2^7 z \log_2 (\alpha^3)=2^7 z P+2^7 z u,
$$
so that
$$
\exp_2((2^7 z\log_2 (\alpha^3))=\exp_2(2^7 zP+2^7 zu)=\exp_2(2^7zP)\exp_2(2^7 zu).
$$
We have
$$
\exp_2(y)=1+y+\frac{y^2}{2}+\cdots+\frac{y^n}{n!}+\cdots.
$$
For $\nu_2(y)\ge 2$ and $n\ge 2$ we have 
$$
\nu_2\left(\frac{y^n}{n!}\right)=n\nu_2(y)-\nu_2(n!)\ge n\nu_2(y)-(n-\sigma_2(n))>n(\nu_2(y)-1)\ge \nu_2(y),
$$
where the last inequality holds as it is equivalent to $\nu_2(y)\ge n/(n-1)$, which is so since $\nu_2(y)\ge 2\ge n/(n-1)$ for all $n\ge 2$. In the above, $\sigma_2(n)$ is the sum of the digits of $n$ in base $2$. It then follows that 
$
\exp_2(y)\equiv 1\pmod {2^{\nu_2(y)}},
$
provided $\nu_2(y)\ge 2$. Hence,
$
\exp_2(2^7 zu)\equiv 1\pmod {2^{7+\nu_2(u)}}\equiv 1\pmod {2^{128}}.
$
This means that
$$
\exp_2(2^7 z\log_2 \alpha^3)\equiv \exp_2(2^7 z P)\pmod {2^{128}}\equiv \sum_{k\ge 0} \frac{(2^7 z P)^k}{k!}\pmod {2^{128}}.
$$
Since $1=\nu_2(\log_3 \alpha^3)=\nu_2(P+u)$ and $\nu_2(u)\ge 121$ is large, we get that $\nu_2(P)=1$. Further,
$$
\nu_2\left(\frac{(2^7 z P)^k}{k!}\right)=k\nu_2(2^7 z P)-\nu_2(k!)\ge (7+\nu_2(P))k-(k-\sigma_2(k))> 7k,
$$
since $\sigma_2(k)\ge 1$ and $\nu_2(P)\ge 1$, so it follows that the above numbers are at least $7\cdot 19=133>128$ for $k\ge 19$. Thus, we may truncate the series at $k=18$, so write
$$
\exp_2(2^7 z\log_2(\alpha^3))=\sum_{k=0}^{18} \frac{(2^7 zP)^k}{k!}\pmod {2^{128}}.
$$
The same argument works with $\alpha$ replaced by $\beta$, so we may write 
\begin{equation}
	\label{eq:Q}
	Q:=-\sum_{n=1}^{120} \frac{(1-\beta^3)^n}{n}
\end{equation}
and then
$$
\exp_2(2^7 z \log_2 (\beta^3))=\sum_{k=0}^{18} \frac{(2^7 zQ)^k}{k!}\pmod {2^{128}}.
$$
Thus,
$$
L_{n+17}-L_n=\sum_{k=0}^{18} \frac{(\alpha^{17}-1)\alpha^{71} (2^7 zP)^k+(\beta^{17}-1)\beta^{71}(2^7 z Q)^k}{k!}\pmod {2^{128}}.
$$
The right--hand side above is a polynomial of degree $18$ in $z$ whose coefficients are rational numbers which are $2$--adic integers (that is, the numerators of those rational numbers are always odd). We will show that in our range the above expression is never $0$ modulo $2^{128}$. This will show that
$\nu_2(L_{n+17}-L_n)<128$ for $n<3.2\cdot 10^{33}.$

We need to find these numbers which is not so easy in SageMath as $P$ and $Q$ involve large powers of $\alpha$ and $\beta$. Nevertheless, we can compute $A:=P+Q$ and $B:=PQ$. Next, the coefficients 
$$
u_k:=(\alpha^{17}-1)\alpha^{71} P^k+(\beta^{17}-1)\beta^{71} Q^k
$$
form a linearly recurrence sequence of recurrence 
$$
u_{k+2}=Au_{k+1}-Bu_k\qquad {\text{\rm for}}\qquad k\ge 0,
$$
with $u_0$ and $u_1$ obtained from $u_k=(\alpha^{17}-1)\alpha^{71} P^k+(\beta^{17}-1)\beta^{71} Q^k$ when $k=0,1$ respectively.
So, we can compute all the remaining ones iteratively and look at the polynomial
$$
f(z):=\sum_{k=0}^{18} (2^7 z)^k \frac{u_k}{k!}\pmod {2^{128}}.
$$
All coefficients $u_k/k!$ are $2$--adic integers and we can reduce them modulo $2^{128}$. At this step we obtained a polynomial  in ${\mathbb Z}/(2^{128}{\mathbb Z})[z]$.
And we need to find $z$ such that this polynomial is $0\pmod {2^{128}}$. For this we do it one step at a time. Namely, with $2^7 z$ we reduce $f(z)$ modulo $2^{10},~2^{11},~2^{12},~2^{13},~$ and so on and guess what the corresponding digit of $z$ must be in the next power of $2$ ($0$ or $1$) to make the above polynomial divisible by higher and higher powers of $2$ (this is just Hensel's lemma at work). We end up with $
z=2+2^2+2^4+2^5+2^6+2^7+2^8+2^9+\cdots$, and after extracting digits up to $2^{101}$, namely writing 
$$
z=2+2^2+2^4+2^5+2^6+2^7+2^8+2^9+\cdots+2^{96}+2^{100}+2^{101}+2^{102} t
$$
and reducing $f(z)$ modulo $2^{115}$, we get $
2^{111} (8 + 9 t)\pmod {2^{115}}$.
Notice that we should choose $t$ to be a multiple of $8$, so 
$$
n\ge 3\cdot 2^7 \cdot (\cdots +2^{102}\cdot 8)=2^{112}\cdot 3>3.2\cdot 10^{33}.
$$
What this argument shows is that in effect, $\nu_2(L_{n+17}-L_n)<114$.

A similar analysis was performed for the other values of $d$ as well as for $p=3$. In the case of $p=3$, the period of $(L_n)_{n\ge 0}$ modulo $3^{k+1}$ is $4\cdot 3^k$, so we work $3-$adically with $\log_3 (\alpha^4)$ and $\log_3 (\beta^4)$. See Appendix 4 for the automated process in SageMath, for each $p=2,3$ and each $d$. In all cases when $n-n_1\ge 5$ is odd, we obtained that $\nu_p(L_{n}-L_{n_1})<114$.

Hence, in all cases, we conclude that
\[x_{min}< 124~~\text{and}~~y_{min}< 79 .\]	
Lastly, we find a smaller upper bound for $n$. If we write $c_X$ for the upper bound of $X-X_1$, then 
\begin{align*}
X=X_2+(X_1-X_2)+(X-X_1)&<x_{min}\log 2 +y_{min}\log 3+2c_X,\\
x\log 2 +y\log 3&<124\log 2 +79\log 3+2\cdot 157<487.
\end{align*}
Hence, $x<703$ and $y<444$. On the other hand, Lemma \ref{lem2.9} implies that $n\log\alpha<X<487,$
so that $n<1013$. This contradicts the working assumption that $n>1500$. Thus, Theorem \ref{1.2b} is proved.\qed

\section{Proof of Theorem \ref{1.2c}}\label{sec5}
In this section, we prove Theorem \ref{1.2c}.
\subsection{The case $n_1<n\le 500$.}\label{subsec5.1}
As before, we cannot have $n=n_1$ as this leads to $x=x_1$, $y=y_1$ and obtain the same representation of $c$. So, without loss of generality, we assume $n>n_1$. Then, for $c<0$, we can write 
\begin{align*}
L_n-2^x3^y=L_{n_1}-2^{x_1}3^{y_1}=L_{n_2}-2^{x_2}3^{y_2}~(=c<0).
\end{align*}
Equating the first two equations above, we have $$2^x3^y-2^{x_1}3^{y_1}=L_n-L_{n_1}<L_n\le L_{500}<1.001\alpha^{500}<10^{105},$$ 
where we used relation \eqref{2.2}. We therefore have that 
\begin{align*}
	\left|2^x3^y-2^{x_1}3^{y_1}\right|<10^{105},
\end{align*}
and dividing through by $2^x3^y$, we get
\begin{align}\label{eq1}
	|\tau|=\left|1-2^{x_1-x}3^{y_1-y}\right|\le \dfrac{10^{105}}{2^x3^y},
\end{align}
where we note that $\max\{|x-x_1|,|y-y_1|\}\log 2\le x\log 2+y\log 3$. Clearly, $\tau\ne 0$ otherwise we would have $x=x_1$ and $y=y_1$. This would imply $n=n_1$ contradicting the assumption that $n_1<n$. Moreover, $2$ and $3$ are linearly independent over $\mathbb{Q}$, so we apply Theorem \ref{thm:LMN} on $\tau$ with $\gamma_1:=2$, $\gamma_{2}:=3$, $b_1:=x_1-x$ and $b_2:=y_1-y$. Further, $\mathbb{K}:=\mathbb{Q}$, $D=1$ and so we can write $\log A_i:=2$, for $i=1,2$. Also, we have
\begin{align*}
b'&:=\frac{|b_1|}{D\log A_2}+\frac{|b_2|}{D\log A_1}=\dfrac{1}{2}\left(|x_1-x|+|y_1-y|\right)\le \dfrac{1}{2}\cdot2\max\{|x-x_1|,|y-y_1|\}
\le \dfrac{x\log 2+y\log 3}{\log 2}.
\end{align*}
Therefore, $
	\log |\tau|\ge -24.34 \left(\max\left\{\log b'+0.14,~21,~0.5\right\}\right)^2\cdot 2\cdot 2
	>	-98 \left(\max\left\{\log b'+0.14,~21\right\}\right)^2$.
Using the upper bound from relation \eqref{eq1}, we get 
\begin{align*}
	x\log 2+y\log 3-\log 10^{105}&<98 \left(\max\left\{\log b'+0.14,~21\right\}\right)^2;\\
	x\log 2+y\log 3&<98 \left(\max\left\{\log b'+0.14,~21\right\}\right)^2 + 242.
\end{align*}
Now, if $\max\left\{\log b'+0.14,~21\right\}=21$, then $x\log 2+y\log 3<43460$ so that $x<62700$ and $y<39559$. If
\begin{align*}
	\max\left\{\log b'+0.14,~21\right\}&=\log b'+0.14
	\le\log \left(\dfrac{x\log 2+y\log 3}{\log 2}\right)+0.14
	< 2\log \left(x\log 2+y\log 3\right),
\end{align*}
we get that 
$x\log 2+y\log 3<98 \left( 2\log \left(x\log 2+y\log 3\right)\right)^2 + 242<94864 \left( \log \left(x\log 2+y\log 3\right)\right)^2$. Applying Lemma \ref{Guz} with $z:=x\log 2+y\log 3$, $m:=2$ and $T:=94864 >(4s^2)^s=256$, we get 
	$$x\log 2+y\log 3<5\cdot 10^7,$$
	implying that $x<7.3\cdot 10^7$ and $y<4.6\cdot 10^7$.
	
Therefore, in all cases, $x<7.3\cdot 10^7$ and $y<4.6\cdot 10^7$. These bounds on $x$ and $y$ are too large to allow direct computation, so we reduce them. We go back to equation \eqref{eq1} and assume $10^{105}/2^x3^y<1/2$, so that Lemma \ref{lem2.1} gives
\begin{align*}
	|\log(1+\tau)|=|(x_1-x)\log 2+(y_1-y)\log 3| <\dfrac{1.5\cdot 10^{105}}{2^x3^y}.
\end{align*}
Next, we divide the above equation by $|y_1-y|\log 2$ and get
\begin{align*}
	\left|\dfrac{\log 3}{\log 2} - \dfrac{x-x_1}{y_1-y} \right|<\dfrac{2.2\cdot 10^{105}}{2^x3^y|y_1-y|},
\end{align*}
since $y_1$ and $y$ are distinct. Note that $y_1$ and $y$ are indeed distinct since if they were not, then $0.5\le |\tau|=\left|1-2^{x_1-x}\right|\le 10^{105}/2^x3^y<0.5$, a contradiction.   By Lemma \ref{lem:Leg} with $\mu:=\dfrac{\log 3}{\log 2} $ and $M:=10^{8}$, we have 
\begin{align*}
	\dfrac{1}{(a(M)+2)(y_1-y)^2}<\left|\dfrac{\log 3}{\log 2} - \dfrac{x-x_1}{y_1-y} \right|<\dfrac{2.2\cdot 10^{105}}{2^x3^y|y_1-y|},
\end{align*}
where $a(M)=55$ (in fact, $q_{17}>10^{8}$ and $\max\{a_k: 0\le k\le 17\}=55$). The above inequality gives
\begin{align*}
	\dfrac{1}{(55+2)(y_1-y)^2}&<\dfrac{2.2\cdot 10^{105}}{2^x3^y|y_1-y|},
\end{align*}
so that 
\begin{align*}
	2^x3^y&<1.3\cdot 10^{107}|y_1-y|<1.3\cdot 10^{107}\cdot 4.6\cdot 10^7<6\cdot 10^{114},
\end{align*}
where we have used the upper bound $y<4.6\cdot 10^7$. Taking logarithms of both sides and simplifying, we get $x<382$
 and $y<241$. Note that if the above assumption that $10^{105}/2^x3^y<1/2$ is violated, then $10^{105}/2^x3^y\ge 1/2$ implies that $2^x3^y\le 2\cdot 10^{105}$, so that $x<350$ and $y<221$.
 
 Therefore, in all cases, we have that for $c<0$, if $n_1<n\le 500$, then $x<382$ and $y<241$. A computer search in SageMath 9.5, looks for all $c<0$ with at least three representations of the form $L_n-2^x3^y$, and returns only the solutions given in Theorem \ref{1.2c}. This computation lasted for about 1 hour on an 8GB RAM Intel\textsuperscript{®} Core\textsuperscript{TM} laptop.
 
\subsection{The case $n_1\le 500$ and $n>500$.}
In this subsection, we assume $n>500$ and determine an absolute bound on $n$ with $n_1\le 500$. The reason for bounding $n$ here follows from the analysis done in Subsection \ref{subsec5.1} which tells us that once $n$ is bounded by a constant, say $K$, then we are sure that $x,y<K$. Note that this case is not restrictive since the roles of $n$ and $n_1$ are interchangeable here.

Since $c<0$, then $L_n-2^x3^y<0$ and hence $2^x3^y-L_n>0$. We go back to \eqref{1.2} and rewrite it as
	\begin{align}\label{eq2}
		0<	2^x3^y-L_n&=2^{x_1}3^{y_1}-L_{n_1};\nonumber\\
		2^x3^y-\alpha^n&=2^{x_1}3^{y_1}-\alpha^{n_1}+\beta^n-\beta^{n_1}\nonumber\\
		&<2^{x_1}3^{y_1}+\beta^n-\beta^{n_1};\nonumber\\
		|2^x3^y-\alpha^n|&<2^{x_1}3^{y_1}+1,
	\end{align}
	where the last inequality holds for all $n>500$, by similar arguments used to obtain \eqref{4.5}. Since $n_1\le 500$, then $x_1<382$ and $y_1<241$ by a similar analysis done in Subsection \ref{subsec5.1} considering the situation $n_2<n_1\le 500$.
Thus, relation \eqref{eq2} becomes
\begin{align*}
	|2^x3^y-\alpha^n|&<2^{x_1}3^{y_1}+1\le 3^{x_1+y_1}+1<3^{624},
\end{align*}
and dividing through by $\alpha^n$ gives
\begin{align}\label{eq3}
	|2^x3^y\alpha^{-n}-1|<3^{624}\alpha^{-n}.
\end{align}
The left--hand side of \eqref{eq3} is the same as the left--hand side of \eqref{4.2}, so we apply Theorem \ref{thm:Mat} to the left--hand side of \eqref{eq3} with the same data as used in the proof of Lemma \ref{lem4.1}. We get
\begin{align*}
	n\log \alpha-624\log 3<1.62\cdot 10^{12} \log n,
\end{align*}
so that $n<3.4\cdot 10^{12} \log n$. Let $ z:=n $, $ m:=1 $ and $T:=3.4\cdot 10^{12}>4$, 
then Lemma \ref{Guz} gives $n< 2\cdot 10^{14}.$ Again, this bound is too large to allow direct computation. We reduce it. To do so, we consider \eqref{eq3} with the assumption that $3^{624}\alpha^{-n}<1/2$ and take logarithms, that is
\[
\left|n\log \alpha-x\log 2-y\log3\right|<1.5\cdot 3^{624}\alpha^{-n}.
\]
Again, we consider the approximation lattice
$$ \mathcal{A}=\begin{pmatrix}
	1 & 0  & 0 \\
	0 & 1 & 0 \\
	\lfloor C\log 2\rfloor & \lfloor C\log 3\rfloor& \lfloor C\log\alpha \rfloor
\end{pmatrix} ,$$
with $C:= 10^{43}$ and choose $v:=\left(0,0,0\right)$. Now, by Lemma \ref{lem2.5}, we get  $|\Lambda_2|>c_1:=4.1\cdot 10^{-16}$ and hence $c_2:=4.91\cdot 10^{15}$. Moreover, by a similar analysis as done in Subsection \ref{subsec5.1} but with $n<2\cdot 10^{14}$, we have
$
x,y,n<A_1=A_2=A_3=2\cdot 10^{14}.
$
So, Lemma \ref{lem2.6} gives $S=8\cdot 10^{28}$ and $T=3.1\cdot 10^{14}$. Since $c_2^2\ge T^2+S$, then choosing $c_3:=1.5\cdot 3^{624}$ and $c_4:=\log\alpha$ gives $n< 1557$.
Note that if the above assumption that $3^{624}\alpha^{-n}<1/2$ is violated, then we would get $n<1427$, a smaller bound.

At this point, we have that $n_1\le 500$ and $n<1557$. We go back and do similar analysis as in Subsection \ref{subsec5.1} to obtain bounds on $x$ and $y$ in this case. As before, we write $$2^x3^y-2^{x_1}3^{y_1}=L_n-L_{n_1}<L_n< L_{1557}<1.001\alpha^{1557}<10^{326},$$ 
where we used relation \eqref{2.2}. We therefore have that 
$
	\left|2^x3^y-2^{x_1}3^{y_1}\right|<10^{326},
$
and dividing through by $2^x3^y$ gives
\begin{align}\label{eq4}
	|\tau_1|=\left|1-2^{x_1-x}3^{y_1-y}\right|\le \dfrac{10^{326}}{2^x3^y}.
\end{align}
We apply Theorem \ref{thm:LMN} to the left--hand side of \eqref{eq4} with the same data used on $\tau$ in \eqref{eq1}, to get
$
\log |\tau_1|> -98 \left(\max\left\{\log b'+0.14,~21\right\}\right)^2$. Now, using the upper bound from relation \eqref{eq4}, we get 
\begin{align*}
	x\log 2+y\log 3-\log 10^{326}&<98 \left(\max\left\{\log b'+0.14,~21\right\}\right)^2;\\
	x\log 2+y\log 3&<98 \left(\max\left\{\log b'+0.14,~21\right\}\right)^2 + 751.
\end{align*}
If $\max\left\{\log b'+0.14,~21\right\}=21$, then $x\log 2+y\log 3<43969$ so that $x<63434$ and $y<40023$. If
$
	\max\left\{\log b'+0.14,~21\right\}=\log b'+0.14
	< 2\log \left(x\log 2+y\log 3\right),
$ as deduced in Subsection \ref{subsec5.1}, we get that 
$x\log 2+y\log 3<98 \left( 2\log \left(x\log 2+y\log 3\right)\right)^2 + 751<294392 \left( \log \left(x\log 2+y\log 3\right)\right)^2$. Applying Lemma \ref{Guz} with $z:=x\log 2+y\log 3$, $m:=2$ and $T:=294392 >(4s^2)^s=256$, we get 
$$x\log 2+y\log 3<2\cdot 10^8,$$
implying that $x<3\cdot 10^8$ and $y<1.9\cdot 10^8$.
So, in all cases, $x<3\cdot 10^8$ and $y<1.9\cdot 10^8$. These bounds on $x$ and $y$ are too large, so we reduce them. We go back to equation \eqref{eq4} and assume for a moment that $10^{326}/2^x3^y<1/2$, so that Lemma \ref{lem2.1} gives
\begin{align*}
	|\log(1+\tau_1)|=|(x_1-x)\log 2+(y_1-y)\log 3| <\dfrac{1.5\cdot 10^{326}}{2^x3^y}.
\end{align*}
Next, we divide the above equation by $|y_1-y|\log 2$ and apply Lemma \ref{lem:Leg} with $\mu:=\dfrac{\log 3}{\log 2} $ and $M:=10^{9}$ as before to get 
\begin{align*}
	\dfrac{1}{(a(M)+2)(y_1-y)^2}<\left|\dfrac{\log 3}{\log 2} - \dfrac{x-x_1}{y_1-y} \right|<\dfrac{2.2\cdot 10^{326}}{2^x3^y|y_1-y|},
\end{align*}
where $a(M)=55$ (in fact, $q_{20}>10^{9}$ and $\max\{a_k: 0\le k\le 20\}=55$). The above inequality gives
\begin{align*}
	\dfrac{1}{57(y_1-y)^2}&<\dfrac{2.2\cdot 10^{326}}{2^x3^y|y_1-y|},
\end{align*}
so that 
$2^x3^y<3\cdot 10^{336}$.
Taking logarithms both sides and simplifying, we get $x<1118$
and $y<706$. Moreover, if the above assumption that $10^{326}/2^x3^y<1/2$ is violated, then $10^{326}/2^x3^y\ge 1/2$ implies that $x<1084$ and $y<684$.

Therefore, in this subsection, we have that for $c<0$, if $n_1\le 500$ and $n>500$, then $n<1557$, $x<1118$ and $y<706$. A computer search in SageMath 9.5, looks for all $c<0$ with at least three representations of the form $L_n-2^x3^y$, and returns no solutions in these ranges. 

\subsection{An absolute upper bound for $n>n_1>500$ in the case that $c \in-\mathbb{N}$.}
For the remaining part of Section \ref{sec5}, we assume $n>n_1>500$. We first prove the following result, keeping in mind that $X=x\log2+y\log 3$ and $X_1=x_1\log2+y_1\log 3$.
\begin{lemma}\label{lem5.1}
	Let $c<0$ such that $(n,x,y)$ and $(n_1,x_1,y_1)$ satisfy \eqref{1.2} with $n>n_1>500$, then
	\[	X-X_1<1.7\cdot 10^{12} \log n. \]
\end{lemma}
\begin{proof}
Since $c<0$, then $L_n-2^x3^y\le -1$ and hence $2^x3^y-L_n>0$. We go back to \eqref{1.2} and rewrite it as
\begin{align}\label{5.1}
0<	2^x3^y-L_n&=2^{x_1}3^{y_1}-L_{n_1},\nonumber\\
	2^x3^y-\alpha^n&=2^{x_1}3^{y_1}-\alpha^{n_1}+\beta^n-\beta^{n_1}\nonumber\\
	&<2^{x_1}3^{y_1}+\beta^n-\beta^{n_1},\nonumber\\
	|2^x3^y-\alpha^n|&<2^{x_1}3^{y_1}+1,
\end{align}
where the last inequality holds for all $n>500$, by similar arguments used to obtain \eqref{4.5}. So, dividing through both sides of \eqref{5.1} by $2^x3^y$, we conclude that
\begin{align}\label{5.2}
	\left|2^{-x} 3^{-y}\alpha^{n}-1\right| \le 2\exp\{-(X-X_1)\}.
\end{align}
%Again like in equation \eqref{4.5}, if $X - X_1>1.4$, then $2\exp \left(-(X-X_1)\right)<1/2$. 
We apply Theorem \ref{thm:Mat} on the left--hand side of \eqref{5.2}. Let $\Gamma_2:=2^{-x} 3^{-y}\alpha^{n}-1=e^{\Lambda_2}-1$. Notice that $\Lambda_2\ne 0$, otherwise we would have $2^x 3^y=\alpha^{n}$, a contradiction since $\alpha^n\not\in\mathbb{Q}$, for any positive integer $n$. So, $\Lambda_2\ne 0$. We use the field $\mathbb{Q}(\sqrt{5})$ with
degree $D = 2$. Here, $t := 3$,
\begin{equation}\nonumber
	\begin{aligned}
		\gamma_{1}&:=2, ~~\gamma_{2}:=3, ~~\gamma_{3}:=\alpha,\\
		b_1&:=-x, ~~b_2:=-y, ~~b_3:=n.
	\end{aligned}
\end{equation}
$B \geq  \max\{x,y,n\}=n$, so define $B:=n$. Again, 
$A_1 = Dh(\gamma_{1}) = 2\log2$,
$A_2 = Dh(\gamma_{2}) = 2\log3$, and
$A_3 = Dh(\gamma_{3}) = 2\cdot\frac{1}{2} \log\alpha = \log \alpha$. Then, by Theorem \ref{thm:Mat},
\begin{align}\label{5.3}
	\log |\Gamma_2| &> -1.4\cdot 30^6 \cdot3^{4.5}\cdot 2^2 (1+\log 2)(1+\log n)(2\log 2)(2\log3)(\log \alpha)\nonumber\\
	&> -1.66\cdot 10^{12} \log n.
\end{align}
Comparing \eqref{5.2} and \eqref{5.3}, we get
\begin{align}\label{5.4}
	X-X_1<1.7\cdot 10^{12} \log n.
\end{align}
This proves Lemma \ref{lem5.1}.
\end{proof}

Next, we prove the following.
\begin{lemma}\label{lem5.2}
	Let $c<0$ such that $(n,x,y)$, $(n_1,x_1,y_1)$ and  $(n_2,x_2,y_2)$ satisfy \eqref{1.2} with $n>n_1>n_2$ and $n_1>500$, then
	\[	n-n_1<8\cdot 10^{13}(\log n)^2. \]
\end{lemma}
\begin{proof}
	Since we assume a third solution $(n_2,x_2,y_2)$ to \eqref{1.2}, then Lemma \ref{lem2.9} also holds for $n_1>n_2$. By hypothesis, $n_1>500$, so 
		\[
	0.38\alpha^{n_1}<\exp(X_1)=2^{x_1}3^{y_1}<2\alpha^{n_1}\left(n_1\log \alpha\right)^{60\log (n_1\log\alpha)}.
	\]
We can then rewrite \eqref{5.1} with the above inequality and get
\begin{align*}
	|2^x3^y-\alpha^n|&<2^{x_1}3^{y_1}+1
	<2\alpha^{n_1}\left(n_1\log \alpha\right)^{60\log (n_1\log\alpha)}+1\\
	&<2.01\alpha^{n_1}\left(n_1\log \alpha\right)^{60\log (n_1\log\alpha)}.
\end{align*}
Dividing through by $\alpha^n$, we obtain
\begin{align}\label{5.5}
	\left|2^{x} 3^{y}\alpha^{-n}-1\right| \le 2.01\alpha^{-(n-n_1)}\left(n_1\log \alpha\right)^{60\log (n_1\log\alpha)}.
\end{align}
%To apply Theorem \ref{thm:Mat} on the left-hand side of \eqref{5.5}, we have to make sure that $$2.01\alpha^{-(n-n_1)}\left(n_1\log \alpha\right)^{60\log (n_1\log\alpha)}<0.5.$$
%This is certainly true if 
%\begin{align*}
%\alpha^{n-n_1}&>\dfrac{2.01\left(n_1\log \alpha\right)^{60\log (n_1\log\alpha)}}{0.5},\\
%\text{or}~~n-n_1&>\dfrac{\log 4.02 +60\left(\log(n_1\log\alpha)\right)^2}{\log\alpha}>4759:=A^*,
%\end{align*}
%for all $n_1>1000$.
Let $\Gamma_3:=2^{x} 3^{y}\alpha^{-n}-1=e^{\Lambda_3}-1$. By the same arguments and data used in the proof of Lemma \ref{lem5.1} above, we conclude by Matveev's Theorem \ref{thm:Mat} that
\begin{align*}
(n-n_1)\log\alpha-\log 2.01 -60\left[\log (n_1\log\alpha)\right]^2&<1.7\cdot 10^{12} \log n;\\
n-n_1&<\dfrac{60\left[\log (n_1\log\alpha)\right]^2+\log 2.01 +1.7\cdot 10^{12} \log n}{\log\alpha}\\
&<8\cdot 10^{13}(\log n)^2.
\end{align*}
%Since this bound is larger than $A^*=4759$, 
Lemma \ref{lem5.2} is proved.
\end{proof}

Next, we retain the notation $x_{min}:=\min\{x,x_1\}$ and $y_{min}:=\min\{y,y_1\}$ and prove the following result.
\begin{lemma}\label{lem5.3}
	Let $c<0$ such that $(n,x,y)$, $(n_1,x_1,y_1)$ and  $(n_2,x_2,y_2)$ satisfy \eqref{1.2} with $n>n_1>n_2$ and $n_1>500$, then
either	\[	x_{min},~y_{\min}<  2.4\cdot 10^{17}\left(\log n\right)^4, \]
or $n<60000$.
\end{lemma}
\begin{proof}
	We follow the same arguments as in the proof of Lemma \ref{lem4.3} but using the bounds from Lemma \ref{lem5.2}. In particular, we consider again the $p-$adic valuations \eqref{4.8} and \eqref{4.9}.

	First, we estimate $v_p\left(\beta^{n-n_1}-1\right)$ for $p=2,3$. By Lemma \ref{lem2.4},	
	\begin{align*}
		v_p\left(\beta^{n-n_1}-1\right)&\le 1+\dfrac{\log\left(n-n_1\right)}{\log p}
		< 1+\dfrac{\log\left(8\cdot 10^{13} (\log n)^2\right)}{\log p}
		<1+\dfrac{6 \log n}{\log p},
	\end{align*}
	under the assumption that $n > 500$. If we assume that $n - n_1$ is even or $n - n_1=1,3$ (see Lemma \ref{lem2.3}), we have		
	\begin{align*}
		v_p\left(\left(\dfrac{\alpha}{\beta}\right)^{n_1} \dfrac{\alpha^{n-n_1}-1}{\beta^{n-n_1}-1}+1 \right)&=v_p\left(\pm \alpha^{2n_1+n-n_1}+1\right)
		<1+\dfrac{\log 2n}{\log p},
	\end{align*}	
	by Lemma \ref{lem2.4}. Therefore, we get inequalities
	\[
	x_{min}<1+\dfrac{\log 2n}{\log 2}+1+\dfrac{6 \log n}{\log 2}<11\log n,
	\]
	and 	
	\[
	y_{min}<1+\dfrac{\log 2n}{\log 3}+1+\dfrac{6 \log n}{\log 3}<7\log n.
	\]
	
	For the case when $n - n_1\ge 5$ is odd, we apply Lemma \ref{lem:Bug} to the right hand side \eqref{4.8} and \eqref{4.9}, respectively. Here, we note that the only difference between the cases $c>0$ and $c<0$ is that we have different upper bounds for $n-n_1$, which results in different upper bounds for $h'(\gamma_{2})$. Indeed, 
	\begin{align*}
	h'(\gamma_{2}) &\le 2h\left( \alpha^{n-n_1}-1\right) <(n-n_1)\log\alpha +\log 4<8\cdot 10^{13}(\log n)^2\log \alpha+\log 4
	<4\cdot10^{13}(\log n)^2.
	\end{align*}
Therefore,
	\[
	E'=\dfrac{b_1}{h'(\gamma_{2})}+\dfrac{b_2}{h'(\gamma_{1})}<\dfrac{n_1}{h'(\gamma_{1})}+\dfrac{1}{h'(\gamma_{2})}<\dfrac{n}{h'(\gamma_{1})}.
	\]
	Assuming that $n>60000$, then 
	\begin{align*}
		E&=\max\left\{\log E'+\log\log p+0.4, 10, 10\log p\right\}
		<\log n - \log h'(\gamma_{1})+\log\log p+0.4
		< \log n+1.1,	 
	\end{align*}
	in any case.  Therefore, \eqref{4.8} implies that
	\begin{align*}
		x_{min}&\le \dfrac{24pg}{(p-1)(\log p)^4} E^2D^4   h'(\gamma_{1})h'(\gamma_{2}  )+v_2\left(\beta^{n-n_1}-1\right)\\
		&< \dfrac{24\cdot 2\cdot 3}{(2-1)(\log 2)^4} (\log n)^2\left(1+\dfrac{1.1}{\log 60000}\right)^2\cdot 2^4  ( \log\alpha )\cdot 4\cdot10^{13}(\log n)^2+1+\dfrac{6 \log n}{\log 2}\\
		&< 2.4\cdot 10^{17}\left(\log n\right)^4,
	\end{align*}
	where we have worked under the assumption that $n>60000$ and since $p=2$ implies that $g= 3$.
	
	In a similar way, \eqref{4.9} gives
	\begin{align*}
		y_{min}&\le \dfrac{24pg}{(p-1)(\log p)^4} E^2D^4   h'(\gamma_{1})h'(\gamma_{2}  )+v_3\left(\beta^{n-n_1}-1\right)\\
		&< \dfrac{24\cdot 3\cdot 4}{(3-1)(\log 3)^4} (\log n)^2\left(1+\dfrac{1.1}{\log 60000}\right)^2\cdot 2^4  \left( \frac{1}{2}\log3 \right)\cdot 4\cdot10^{13}(\log n)^2+1+\dfrac{6 \log n}{\log 3}\\
		&< 2.4\cdot 10^{17}\left(\log n\right)^4.
	\end{align*}
This completes the proof of Lemma \ref{lem5.3}.
\end{proof}

Lastly, under this subsection, we deduce an absolute bound for $n$. We prove the following.
\begin{lemma}\label{lem5.4}
	Let $c<0$ and $n>60000$ such that $(n,x,y)$, $(n_1,x_1,y_1)$ and $(n_2,x_2,y_2)$ satisfy \eqref{1.2} with $n>n_1>n_2$ and $n_1>500$, then
	\[
	n<2\cdot 10^{25}   ,~~~x<1.4\cdot 10^{25}   ,~~~y<8.9\cdot 10^{24} .
	\]
\end{lemma}
\begin{proof}
Lemma \ref{lem5.3} tells that out of any two solutions, the minimal $x$ and $y$ are bounded by $2.4\cdot  10^{17}(\log n)^4$. So, out of the three solutions, at most one of them has $x$ which is not bounded by $2.4\cdot 10^{17}(\log n)^4$ and at most one of them has $y$ which is not bounded by $2.4\cdot 10^{17}(\log n)^4$. Hence, at least one of the solutions has both $x$ and $y$ bounded by $2.4\cdot 10^{17}(\log n)^4$, which in particular, shows that 
	$$
	X_2=x_2\log 2+y_2\log 3<(\log 2+\log 3)\cdot 2.4\cdot 10^{17}(\log n)^4<4.31\cdot 10^{17} (\log n)^4.
	$$

	Now, by Lemma \ref{lem5.1} together with Lemma \ref{lem2.9}, we have
	\begin{align}\label{5.6}
		n\log\alpha&<X=X_2+(X_1-X_2)+(X-X_1)\nonumber\\
		&<4.31\cdot 10^{17}\left(\log n\right)^4+1.7\cdot 10^{12}\log n+1.7\cdot 10^{12}\log n\nonumber\\
		&<4.32\cdot 10^{17}\left(\log n\right)^4.
	\end{align} 
%	\[
%	\dfrac{n}{\left(\log n\right)^4}<4.32\cdot 10^{17}.
%	\]
Lastly, we apply Lemma \ref{Guz} to the above inequality with $ z:=n $, $ m:=4\geq 1 $, $T:=4.32\cdot 10^{17}$.
	Notice that since $T>(4\cdot4^2)^4=16777216$,
	then, $$n< 2^4 \cdot 4.32\cdot 10^{17}(\log 4.32\cdot 10^{17})^4 < 2\cdot 10^{25}.$$	
	Further, we have by Lemma \ref{lem2.9} that
	\begin{align*}
		X&<2+n\log\alpha+60\left(\log (n\log\alpha)\right)^2\qquad\text{or, equivalently}\\
		x\log2+y\log3&<2+2\cdot 10^{25}\log\alpha+60\left(\log (2\cdot 10^{25}\log\alpha)\right)^2 <9.7\cdot 10^{24}.
	\end{align*}	
	This gives $x<1.4\cdot 10^{25}$ and $y<8.9\cdot 10^{24}$.
\end{proof}

\subsection{Reduction of the upper bound on $n$ when $c \in -\mathbb{N}$}
As in Subsection \ref{subsec4.2}, we use LLL--reduction methods, the theory of continued fractions and $p-$adic reduction methods due to \cite{PET} to find a rather small bound for $n$, which will conclude the proofs of Theorem \ref{1.2c}.

First, we consider \eqref{5.2} with the assumption that $X-X_1\ge 100$ and take logarithms, that is
\[
|\Lambda_2|=\left|n\log \alpha-x\log 2-y\log3\right|<1.5\exp(-(X-X_1)).
\]
Like before, we consider the approximation lattice
$$ \mathcal{A}=\begin{pmatrix}
	1 & 0  & 0 \\
	0 & 1 & 0 \\
	\lfloor C\log 2\rfloor & \lfloor C\log 3\rfloor& \lfloor C\log\alpha \rfloor
\end{pmatrix} ,$$
with $C:= 10^{76}$ and choose $v:=\left(0,0,0\right)$. Now, by Lemma \ref{lem2.5}, we get  $|\Lambda_2|>c_1:=6.4\cdot 10^{-28}$ and hence $c_2:=2.73\cdot 10^{26}$. Moreover, by Lemma \ref{lem5.4}, we have
\[
x<A_1=1.4\cdot 10^{25},~~y<A_2=8.9\cdot 10^{24},~~n<A_3=2\cdot 10^{25}.
\]
So, Lemma \ref{lem2.6} gives $S=2.8\cdot 10^{50}$ and $T=2.15\cdot 10^{25}$. Since $c_2^2\ge T^2+S$, then choosing $c_3:=1.5$ and $c_4:=1$, we get $X-X_1\le 114$.

Next, we continue with the assumption that $n>n_1$, and consider the inequality
\[
0<2^x3^y-2^{x_1}3^{y_1}=L_n-L_{n_1}<L_n<1.001\alpha^n.
\]
Dividing through by $2^x3^y$ and taking logarithms, we get
\[
\left|(x-x_1)\log 2+(y-y_1)\log3\right|<1.51\dfrac{\alpha^n}{2^x3^y},
\]
where we have assumed that $2^x3^y>\alpha^n$ and applied Lemma \ref{lem2.1}. Next, we divide the above equation by $|y_1-y|\log 2$ and get
\begin{align*}
	\left|\dfrac{\log 3}{\log 2} - \dfrac{x-x_1}{y_1-y} \right|<\dfrac{2.2\alpha^n}{2^x3^y|y_1-y|},
\end{align*}
since $y_1$ and $y$ are distinct as explained before. By Lemma \ref{lem:Leg} with $\mu:=\dfrac{\log 3}{\log 2} $ and $M:=10^{25}$, we have 
\begin{align*}
	\dfrac{1}{(a(M)+2)(y_1-y)^2}<\left|\dfrac{\log 3}{\log 2} - \dfrac{x-x_1}{y_1-y} \right|<\dfrac{2.2\alpha^n}{2^x3^y|y_1-y|},
\end{align*}
where $a(M)=55$ (in fact, $q_{49}>10^{25}$ and $\max\{a_k: 0\le k\le 49\}=55$). Multiplying the above inequality by $|y_1-y|\log 2$ gives
\begin{align*}
	\dfrac{\log 2}{57\cdot8.9\cdot 10^{24}}<\dfrac{\log 2}{(a(M)+2)|y_1-y|}<\left|(x-x_1)\log 2+(y-y_1)\log3\right|,
\end{align*}
so that 
\[
1.3\cdot 10^{-27}<\left|(x-x_1)\log 2+(y-y_1)\log3\right|<1.51\dfrac{\alpha^n}{2^x3^y},
\]
where we have used the upper bound $y<8.9\cdot 10^{24}$.
This gives $2^x3^y<1.2\cdot 10^{27}\alpha^n$. Now, by assuming a third solution to \eqref{1.2}, then
\[
0<2^x3^y-\alpha^n<2^{x_1}3^{y_1}<1.2\cdot 10^{27}\alpha^{n_1},
\]
which gives
\begin{align*}
	\left|2^{x} 3^{y}\alpha^{-n}-1\right| <2.4\cdot 10^{27} \alpha^{-(n-n_1)}.
\end{align*}
Assume that $2.4\cdot 10^{27} \alpha^{-(n-n_1)}<0.5$, which is certainly true if $n-n_1\ge133$. Taking logarithms gives
\begin{align*}
	\left|x\log 2+y\log 3-n\log\alpha\right|\le 3.6\cdot 10^{27}\alpha^{-(n-n_1)}.
\end{align*}
So, we consider the approximation lattice
$$ \mathcal{A}=\begin{pmatrix}
	1 & 0  & 0 \\
	0 & 1 & 0 \\
	\lfloor C\log 2\rfloor & \lfloor C\log 3\rfloor& \lfloor C\log\alpha \rfloor
\end{pmatrix} ,$$
with $C:= 10^{78}$ and choose $v:=\left(0,0,0\right)$. Now, by Lemma \ref{lem2.5}, we get $c_1:= 10^{-29}$ and thus $c_2:=10^{27}$. Moreover, by Lemma \ref{lem5.4}, we still have
$
x<A_1=1.4\cdot 10^{25}$, $y<A_2=8.9\cdot 10^{24}$ and $n<A_3=2\cdot 10^{25}.
$
So, Lemma \ref{lem2.6} gives as before $S=2.8\cdot 10^{50}$ and $T=2.15\cdot 10^{25}$, so that choosing $c_3:=3.6\cdot 10^{27}$ and $c_4:=\log\alpha$, we get $n-n_1\le 375$.

To continue, we proceed as in Subsection \ref{subsec4.2} but with different upper bounds for $n$ and $n-n_1$.
First,	
\begin{align*}
	v_p\left(\beta^{n-n_1}-1\right)&\le 1+\dfrac{\log\left(n-n_1\right)}{\log p} < \left\{ \begin{array}{c}
		10 \quad ;~~ \text{if}\quad  p=2,\\
		~7 \quad ;~~ \text{if}\quad p=3.
	\end{array}
	\right.
\end{align*} 
In the case that $n - n_1$ is even or $n-n_1=1,3$ (see Lemma \ref{lem2.3}), we have		
\begin{align*}
	v_p\left(\left(\dfrac{\alpha}{\beta}\right)^{n_1} \dfrac{\alpha^{n-n_1}-1}{\beta^{n-n_1}-1}+1 \right)&=v_p\left(\pm \alpha^{2n_1+n-n_1}+1\right)
	<1+\dfrac{\log 2n}{\log p}
	<\left\{ \begin{array}{c}
		87 \quad ;~~ \text{if}\quad  p=2,\\
		55 \quad ;~~ \text{if}\quad p=3,
	\end{array}
	\right.
\end{align*}	
which gives
$x_{min}\le 97$ and $y_{min}\le 62 $.

If we consider the case that $n-n_1\ge 5$ is odd, we obtain (as in a similar case of Subsection \ref{subsec4.2}) smaller bounds than we found in the case $n-n_1$ is even or $=1,3$.

Lastly, we find a smaller upper bound for $n$. If we write $c_X$ for the upper bound of $X-X_1$, then 
\begin{align*}
	X=X_2+(X_1-X_2)+(X-X_1)&<x_{min}\log 2 +y_{min}\log 3+2c_X;\\
	x\log 2 +y\log 3&<97\log 2 +62\log 3+2\cdot 114<364.
\end{align*}
Hence, $x<526$ and $y<332$. On the other hand, Lemma \ref{lem2.9} implies that $n\log\alpha<X<364,$
so that $n<757$. 

Consequently, the current upper bound for $n$ remains too high to reach a contradiction. Therefore, we repeat the entire reduction procedure, this time using significantly reduced upper limits: $n<757$, $x< 526$ and $y< 332$. We obtain $X-X_1<12$ in the first step, $n-n_1<157$ in the next step, $x_{min}< 20$ and $y_{min}< 14 $ in the following step and consequently $X<54$ which gives $n<113$.
This contradicts the working assumption that $n>500$. Hence, Theorem \ref{1.2c} is proved.\qed

\section*{Acknowledgments} 
The first author extends profound gratitude to the Eastern Africa Universities Mathematics Programme (EAUMP) for their generous support in funding his doctoral studies. The development of this paper significantly benefited from the time spent during his visit to Wits University, Johannesburg. He wishes to express his sincere appreciation for the exceptional hospitality and conducive academic environment provided by the University. Both the first and last co--authors are immensely grateful for the partial support received through the CoEMaSS Grant \#2024-029-NUM from Wits University, Johannesburg, South Africa, which played a pivotal role in this research. Additionally, the last author acknowledges with thanks the fellowship at STIAS. The support and hospitality offered by STIAS were invaluable and greatly appreciated, contributing significantly to the progress of this work.

\section*{Addresses}
$ ^{1} $ Department of Mathematics, School of Physical Sciences, College of Natural Sciences, Makerere University, Kampala, Uganda

Email: \url{hbatte91@gmail.com}

Email: \url{mahadi.ddamulira@mak.ac.ug}
 
Email: \url{juma.kasozi@mak.ac.ug}

\vspace{0.35cm}
\noindent 
$ ^{2} $ School of Mathematics, Wits University, Johannesburg, South Africa and Centro de Ciencias Matem\'aticas UNAM, Morelia, Mexico 

Email: \url{Florian.Luca@wits.ac.za}
\newpage
\section*{Appendices}
\subsection*{Appendix 1}\label{app1}
\begin{verbatim}
	from sage.all import *
	
	# Precompute the Lucas numbers up to L_30
	lucas_sequence = [2, 1]
	for n in range(2, 31):
	   lucas_sequence.append(lucas_sequence[n - 1] + lucas_sequence[n - 2])
	
	# Define the Diophantine equation function
	def solve_diophantine(lucas_n, x, y):
	   return lucas_n - 2**x * 3**y
	
	# Iterate through the range of values for n, x, and y
	solution_set = []
	for n in range(31):
	  lucas_n = lucas_sequence[n]
	  for x in range(31):
	    for y in range(31):
	       if solve_diophantine(lucas_n, x, y) == 0:
	            solution_set.append((n, x, y))
	
	# Print the solutions
	for solution in solution_set:
	    print(solution)
\end{verbatim}
\subsection*{Appendix 2}\label{app2}
\begin{verbatim}
	from sage.all import *
	from collections import defaultdict
	
	# Precompute the Lucas numbers up to L_1500
	lucas_sequence = [2, 1]
	for n in range(2, 1501):
	    lucas_sequence.append(lucas_sequence[n - 1] + lucas_sequence[n - 2])
	
	# Create a dictionary to store representations of the form L_n - 2^x 3^y
	representations = defaultdict(list)
	
	# Iterate through the range of values for n, x, and y
	for n in range(1501):
	    Ln = lucas_sequence[n]
	    for x in range(724):  # 0 <= x <= 723
	        for y in range(724):  # 0 <= y <= 723
	            c = Ln - 2**x * 3**y
	            if c > 0:  # Only consider positive values of c
	                representations[c].append((n, x, y))
	
	# Filter tuples with at least 3 representations
	valid_tuples = [(c, reps) for c, reps in representations.items() if len(reps) >= 3]
	
	# Print the valid tuples
	for c, reps in valid_tuples:
	     print(f"For c = {c}:")
	     for rep in reps:
	          print(f"  (n = {rep[0]}, x = {rep[1]}, y = {rep[2]})")
	
\end{verbatim}
\subsection*{Appendix 3}\label{app3}
\begin{verbatim}
	from sage.all import *
	from collections import defaultdict
	
	# Precompute the Lucas numbers up to L_1000
	lucas_sequence = [2, 1]
	for n in range(2, 1001):
	    lucas_sequence.append(lucas_sequence[n - 1] + lucas_sequence[n - 2])
	
	# Create a dictionary to store representations of the form L_n - 2^x 3^y
	representations = defaultdict(list)
	
	# Iterate through the range of values for n, x, and y
	for n in range(1001):
	    Ln = lucas_sequence[n]
	    for x in range(730):  # 0 <= x <= 729
	        for y in range(461):  # 0 <= y <= 460
	           c = Ln - 2**x * 3**y
	           if c < 0:  # Only consider negative values of c
	               representations[c].append((n, x, y))
	
	# Filter tuples with at least 4 representations
	valid_tuples = [(c, reps) for c, reps in representations.items() if len(reps) >= 4]
	
	# Print the valid tuples
	for c, reps in valid_tuples:
	    print(f"For c = {c}:")
	    for rep in reps:
	         print(f"  (n = {rep[0]}, x = {rep[1]}, y = {rep[2]})")
	
\end{verbatim}
\subsection*{Appendix 4}\label{app4}
\begin{verbatim}
# Import necessary libraries
from sage.all import *

# Define the Lucas sequence function using memoization for efficiency
@cached_function
def L(n):
    if n == 0: return 2
    if n == 1: return 1
    return L(n-1) + L(n-2)

# Define the Lucas sequence modulo p^(k+1)
def lucas_mod(n, k, p):
    if n == 0:
       return 2
    elif n == 1:
       return 1
    period = (p + 1) * p^k
    L = [2 % (p^(k+1)), 1 % (p^(k+1))] + [0]*(period-2)
    for i in range(2, period):
       L[i] = (L[i-1] + L[i-2]) % (p^(k+1))
    return L[n % period]

# Approximate sqrt(5) to a high precision
sqrt5_approx = sqrt(RealField(100)(5))

# Loop over p = 2 and p = 3
for p in [2, 3]:
    # Define the p-adic field with precision 128
    Qp_field = Qp(p, prec=128)

    # Check for each odd d in the given range
    d_values_with_n = []
    for d in range(5, 321, 2):
        found_n = False
        for n in range(1001):
            if (lucas_mod(n+d, 7, p) - lucas_mod(n, 7, p)) % (p^8) == 0:
                  d_values_with_n.append(d)
                  found_n = True
                  break

    # Calculate n_0(d) for each d in d_values_with_n
    n0_values = {}
    for d in d_values_with_n:
        for n in range(1, (p + 1) * p^7 + 1):
             if (lucas_mod(n+d, 7, p) - lucas_mod(n, 7, p)) % (p^8) == 0:
                 n0_values[d] = n
                 break

    # Define alpha and beta using the approximated sqrt(5)
    alpha = (1 + sqrt5_approx) / 2
    beta = (1 - sqrt5_approx) / 2

    # Define P and Q
    P = -sum([(1 - alpha^(p+1))^n / n for n in range(1, 121)])
    Q = -sum([(1 - beta^(p+1))^n / n for n in range(1, 121)])

    # Compute A and B
    A = P + Q
    B = P * Q

    # For each d and its corresponding n_0(d), calculate the maximum nu_2
    for d, n_0_d in n0_values.items():
        # Initialize the coefficients u_k
        u = [0]*19
        u[0] = ((alpha^d - 1) * alpha^n_0_d * P^0 + (beta^d - 1) * beta^n_0_d * Q^0)
        u[1] = ((alpha^d - 1) * alpha^n_0_d * P^1 + (beta^d - 1) * beta^n_0_d * Q^1)

        # Compute the remaining coefficients using the recurrence relation
        for k in range(2, 19):
            u[k] = A * u[k - 1] - B * u[k - 2]

        # Initialize z, the current modulus, and highest_power_of_2
        z = 0
        current_modulus = p^10
        highest_power_of_2 = 0

        # Iteratively update z and the modulus
        for i in range(10, 128):
            # Convert each term to a p-adic number and sum them up
            f_z_terms = []
            for k in range(19):
                term = (p^7 * z)^k * u[k] / factorial(k)
                f_z_terms.append(Qp_field(term))
            f_z = sum(f_z_terms) % current_modulus

            # Determine the next binary digit of z
            for digit in [0, 1]:
                  test_z = z + p^i * digit
                  test_f_z_terms = []
                  for k in range(19):
                      term = (p^7 * test_z)^k * u[k] / factorial(k)
                      test_f_z_terms.append(Qp_field(term))
                  test_f_z = sum(test_f_z_terms) % (p^(i + 1))
                  if test_f_z == 0:
                      z = test_z
                      highest_power_of_2 = i
                      break

            # Update the current modulus
            current_modulus = p^(i + 1)

            # Check if z exceeds the upper bound
            if 2^highest_power_of_2 > 3.2e33:
                 break

      # Output the final value of z for each d
      if 2^highest_power_of_2 > 3.2e33:
          print(f"For p = {p}, d = {d}, n_0(d) = {n_0_d}, the maximum v_p is achieved for 
               z = 2^{highest_power_of_2} (exceeds 3.2e33)")
      else:
          print(f"For p = {p}, d = {d}, n_0(d) = {n_0_d}, the maximum v_p is achieved for 
               z = 2^{highest_power_of_2} (does not exceed 3.2e33)")

\end{verbatim}

\end{document}